\documentclass{cedram-ambp}
\usepackage{amsfonts}
\usepackage{amsmath}
\usepackage{amssymb}
\usepackage{epsfig}

\title
[Basic results on braid  groups]
{Basic results on braid groups}

\alttitle{Resultats basiques dans les groupes de tresses.}

\author
[J. \lastname{Gonz\'alez-Meneses}]
{\firstname{Juan} \lastname{Gonz\'alez-Meneses}}

\address{Departamento de \'Algebra\\
Facultad de Matem\'aticas\\
Universidad de Sevilla\\
Apdo. 1160\\
41080 - Sevilla\\
SPAIN}
\thanks{The author is partially supported by MTM2007-66929 and FEDER}
\email{meneses@us.es}

\altkeywords{Tresses, groupes d'Artin-Tits}


\begin{document}

% The abstract.

\begin{abstract}
These are Lecture Notes of a course given by the author at the French-Spanish School {\it Tresses in Pau}, held in Pau (France) in October 2009. It is basically an introduction to distinct approaches and techniques that can be used to show results in braid groups. Using these techniques we provide several proofs of well known results in braid groups, namely the correctness of Artin's presentation, that the braid group is torsion free, or that its center is generated by the full twist. We also recall some solutions of the word and conjugacy problems, and that roots of a braid are always conjugate. We also describe the centralizer of a given braid. Most proofs are classical ones, using modern terminology. I have chosen those which I find simpler or more beautiful.

\end{abstract}

% French abstract
\begin{altabstract}
Cet article contient les notes d'un course donn\'e par l'auteur \`a l'Ecole Franco-Espagnole {\it Tresses in Pau}, qui a eu lieu \`a Pau (France) en Octobre 2009. Il s'agit essentiellement d'une introduction aux diff\'erents points des vue et techniques qui peuvent \^etre utilisées pour montrer des r\'esultats dans les groupes de tresses. En utilisant ces techniques on montre quelques r\'esultats bien connus dans les groupes de tresses, \`a savoir l'exactitude de la presentation d'Artin, le fait que les groupes de tresses sont sans torsion, ou que son centre est engendr\'e par le {\it full twist}. On rappelle quelques solutions des probl\`emes du mot et de la conjugaison, et aussi que les racines d'une tresse sont toutes conjugu\'ees. On d\'ecrit aussi le centralisateur d'une tresse donn\'ee. La plupart des preuves sont classiques, en utilisant de la terminologie moderne. J'ai choisi celles qui je trouve plus simples ou plus jolies.
\end{altabstract}

% Beware! maketitle *after* abstract
\maketitle

% The beginning of the text.

\tableofcontents

\section{Definitions of braid groups}

The term {\it braid group} was used for the first time by Emil Artin in 1925~\cite{Artin}, although probably these groups were considered for the first time by Hurwitz in 1891~\cite{Hurwitz} as what in modern terminology would be called fundamental groups of configuration spaces of $n$ points in the complex plane (or in a Riemann surface). Magnus in 1934~\cite{Magnus_1934} considered the same group from the point of view of mapping classes. Markoff~\cite{Markoff}  gave a completely algebraic approach. All these points of view were already known to define the same group by that time (see Zariski~\cite{Zariski}).

We shall give several definitions for the braid groups, following the distinct approaches indicated above. The first one, in the spirit of Hurwitz, uses configuration spaces to define a braid as a {\it motion} of points in the plane. The three dimensional representation of this motion gives the usual way in which braids are displayed, and gives the second definition more related to Artin's point of view. Then we will see braid as mapping classes, that is, as homeomorphisms of the punctured disc up to isotopy. Finally, we shall also see Artin's representation of braids as automorphisms of the free group. These distinct points of view provide a considerable amount of approaches and tools to solve problems in braid groups.

\subsection{Pure braids as loops in a configuration space}

Let us start with the first definition. Consider the {\it configuration space} of $n$ ordered distinct points in the complex plane $\mathbb C$. That is,
$$
     M_n =\{(z_1,\ldots,z_n)\in \mathbb C^n;\ z_i\neq z_j,\ \forall i\neq j\}.
$$
Notice that this is a connected space of real dimension $2n$. Actually it is the complement of a family of hyperplanes in $\mathbb C^n$: Namely, if we define the hyperplane $H_{ij}=\{z_i=z_j\}\in \mathbb C^n$ for $1\leq i<j\leq n$, then $M_n=\mathbb C^n \backslash \mathcal D$, where
$$
      \mathcal D = \bigcup_{1\leq i<j\leq n}H_{ij}.
$$
This union of hyperplanes $\mathcal D$ is usually called the {\it braid arrangement}~\cite{Orlik-Terao}, or the {\it big diagonal}.

It is not a good idea to try to visualize the configuration space $M_n$ in $2n$ dimensions. Instead, one just needs to consider $n$ distinct points in $\mathbb C$, which are given in order. That is, $(z_1,z_2,z_3,\ldots,z_n)$ and $(z_2,z_1,z_3\ldots,z_n)$ represent the same set of points in $\mathbb C$, but they are two distinct points in $M_n$.

\begin{definition}
The pure braid group on $n$ strands, $PB_n$, is the fundamental group of $M_n$.
$$
     PB_n=\pi_1(M_n).
$$
\end{definition}

Now try to visualize this definition. A pure braid $\beta\in \pi_1(M_n)$ is a loop in $M_n$
$$
\begin{array}{rccl}   \beta: & [0,1] & \longrightarrow & M_n \\
                            &  t  & \longmapsto & \beta(t) = (\beta_1(t),\ldots,\beta_n(t)),
\end{array}
$$
which starts and ends at the same base point. We then just need to fix a base point of $M_n$, say the $n$-tuple of integers $(1,2,\ldots,n)$, and a pure braid will be represented by a {\it motion} of these points in $\mathbb C$ provided that, at any given moment of the motion, the points are all pairwise distinct. At the end of the motion, each point goes back to its original position. Of course, the loop is defined up to homotopy, so we are allowed to deform the motions in a natural way (provided that two points are never at the same place at the same moment) and we will have equivalent pure braids.

\subsection{Pure braids as collections of strands}

The usual way to visualize a pure braid is by representing the motion of the $n$ points in a three dimensional picture. For every $t\in [0,1]$, we represent the $n$-tuple $\beta(t)=(\beta_1(t),\ldots,\beta_n(t))$ just by drawing, in $\mathbb C\times[0,1]$, the $n$ points $(\beta_1(t),t),\ldots,(\beta_n(t),t)$. In this way a pure braid is represented as in the left hand side of Figure~\ref{Figure_1}. It is a convention for many authors (but not all) that $\mathbb C\times\{0\}$ is represented above $\mathbb C\times\{1\}$. The motion of the point that starts at position $k$ (for $k=1,\ldots,n$) can then be seen as a {\it strand}, which corresponds to the $k$-th projection $\beta_k(t)$ of the $n$-tuple $\beta(t)$, for $t\in[0,1]$. This is called the $k$-th strand of the pure braid: It starts at $(k,0)$ and ends at $(k,1)$.  In this way it is much easier to understand when two motions (or two collections of strands) are homotopic, that is, represent the same pure braid. Indeed, an allowed homotopy is just a continuous deformation within the space of pure braids. In other words, the homotopy of a loop in $M_n$, fixing the endpoints, corresponds to a continuous deformation of the $n$ strands in $\mathbb C\times [0,1]$, provided the endpoints ($(1,0),\ldots,(n,0)$ and $(1,1),\ldots,(n,1)$) are fixed, the strands are pairwise disjoint, and each strand intersects each horizontal plane $\mathbb C\times \{t\}$ exactly at one point, at any given moment of the deformation.

\begin{figure}[ht]
\centerline{\includegraphics{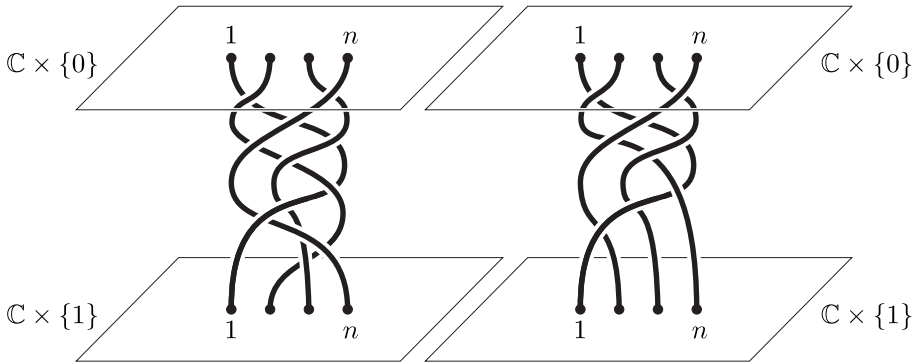}} \caption{A pure braid and a non-pure braid.}
\label{Figure_1}
\end{figure}

Notice that pure braid groups are groups, that is, there is a well defined multiplication of pure braids. This multiplication is given by concatenating loops in $M_n$, which corresponds to performing one motion after the other. In the three dimensional picture of Figure~\ref{Figure_1}, multiplication is just given by stacking braids (and rescaling the vertical direction so that we get $\mathbb C\times [0,1]$ again).

\subsection{Braids, in general}

The definition of general (not necessarily pure) braids can be obtained from the case of pure braids in a very natural way: Braids appear when we do not care about the order in which the $n$ points are considered, so we just care about sets of $n$ distinct points in $\mathbb C$. This configuration space is conceptually more natural than $M_n$, although we use $M_n$ to define it properly: Notice that the symmetric group $\Sigma_n$ acts on $M_n=\mathbb C^n\backslash \mathcal D$ by permuting coordinates. The quotient of $M_n$ by this action is the {\it configuration space of $n$ unordered points in $\mathbb C$}, denoted $N_n=M_n/\Sigma_n$.

\begin{definition}
The braid group on $n$ strands, $B_n$, is the fundamental group of $N_n$.
$$
     B_n=\pi_1(N_n).
$$
\end{definition}

As above, a braid can be represented in $\mathbb C\times [0,1]$ as a collection of $n$ strands, as in the right hand side of Figure~\ref{Figure_1}. The difference with respect to the case of pure braids is that the $k$-th strand (that is, the strand starting at $(k,0)$) does not necessarily end at $(k,1)$, but at $(k',1)$ for some $k'\in \{1,\ldots,n\}$. Again, two braids are considered to be equal if they are homotopic, that is, if one can transform one into the other by a deformation of its strands, with the endpoints fixed, provided the strands are pairwise disjoint and each strand intersects each horizontal plane at a single point, at any given moment during the deformation.

As above, notice that braid groups are groups. The multiplication is given by concatenating loops in $N_n$, which corresponds to performing one motion after the other or, in the three dimensional picture of Figure~\ref{Figure_1}, it corresponds to stacking braids (and rescaling the vertical direction).

\subsection{Braids as mapping classes}

Another well known interpretation of braids consists of considering them as automorphisms of the punctured disc, up to isotopy. More precisely, let $\mathbb D_n$ be the closed disc minus $n$ punctures:
$$
   \mathbb D_n = \mathbb D^2\backslash \{P_1,\ldots,P_n\}.
$$
Let $Homeo^+(\mathbb D_n)$ be the set of orientation preserving homeomorphisms from $\mathbb D_n$ to itself, fixing the boundary pointwise. This set admits the usual compact-open topology. We then have a natural notion of {\it continuous deformation} of an automorphism of $\mathbb D_n$, fixing the boundary and the punctures. We will consider that two automorphisms are equal if one can be transformed into the other by such a continuous deformation. In other words, if we denote by $Homeo_0^+(\mathbb D_n)$ the connected component of $\mbox{id}_{\mathbb D_n}$ in $Homeo^+(\mathbb D_n)$, then we define the {\it mapping class group} of $\mathbb D_n$ as:
$$
   \mathcal M(\mathbb D_n)=Homeo^+(\mathbb D_n)/Homeo_0^+(\mathbb D_n).
$$
With this definition, one has:
$$
     \mathcal  M(\mathbb D_n)\simeq B_n.
$$
Actually, one can consider $\mathbb D_n$ as a closed disc in $\mathbb C$ enclosing the points $\{1,\ldots,n\}$. It is well known that an automorphism of the closed disc $\mathbb D^2$ that fixes the boundary is isotopic to the identity (this is known as Alexander's trick~\cite{Alexander}). Hence, given an element of $\mathcal M(\mathbb D_n)$, one can take a homeomorphism of $\mathbb D_n$ representing it and extend it in a unique way to a homeomorphism $f$ of $\mathbb D^2$. Take an isotopy from $\mbox{id}_{\mathbb D^2}$ to $f$, and trace the motion of the points $\{1,\ldots,n\}$ during the isotopy. This gives a loop in $N_n$ which corresponds to a braid. It is not difficult to show that this gives a well defined map from $\mathcal M(\mathbb D_n)$ to $B_n$, which is a group isomorphism. Hence braids can be seen as mapping classes of the punctured disc.

It is worth mentioning why one should take a closed punctured disc $\mathbb D_n$ instead of the punctured plane $\mathbb C_n=\mathbb C\backslash \{1,\ldots,n\}$, in the above definition. The answer is easy: one does not obtain the same group. Indeed, consider the mapping class group $\mathcal M(\mathbb C_n)$ of isotopy classes of orientation-preserving homeomorphisms of $\mathbb C_n$ (we can consider $\mathbb D_n$ embedded into $\mathbb C_n$, with the same punctures). The first thing to notice is that $\mathbb C_n$ is homeomorphic to the $(n+1)$-times punctured sphere $\mathbb S_{n+1}$, hence a homeomorphism of $\mathbb C_n$ does not necessarily fix the set of punctures $\{1,\ldots,n\}$ (as the ``point at infinity" is considered as any other puncture).

One could then consider $\mathcal M(\mathbb C_n; \{\infty\})$ to be the subgroup of $\mathcal M(\mathbb C_n)$ consisting of automorphisms that fix the point at infinity. But even in this case the resulting group is not isomorphic to $B_n$. Indeed, consider a simple closed curve $C\subset \mathbb D_n$ enclosing the $n$ punctures. A Dehn-twist along this curve (which corresponds to cutting $\mathbb D_n$ along this curve, rotating the boundary of one of the pieces by $360^\circ$ and gluing again)  represents the trivial element in $\mathcal M(\mathbb C_n; \{\infty\})$, but a non-trivial element in $\mathcal M(\mathbb D_n)=B_n$, which is usually denoted $\Delta^2$. Actually, we will see later that the center of $B_n$ is equal to $Z(B_n)=\langle \Delta^2\rangle$, and it is well known that $\mathcal M(\mathbb C_n;\{\infty\})= B_n/ Z(B_n)$. Also, the group $\mathcal M(\mathbb C_n; \{\infty\})$ has torsion, while $B_n$ does not, as we shall see.

\subsection{Standard generators of the braid group}

\begin{figure}[ht]
\centerline{\includegraphics{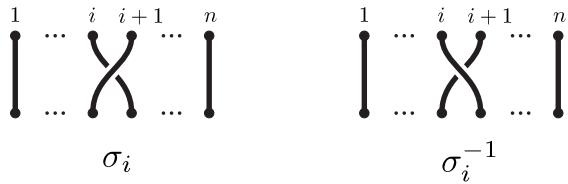}} \caption{Standard (or Artin) generators of $B_n$.}
\label{Figure_2}
\end{figure}

Before giving a last interpretation of braids, we shall describe the standard set of generators, introduced by Artin~\cite{Artin}. If we consider a braid as a collection of $n$ strands in $\mathbb C\times[0,1]$, we usually draw the projection of its strands onto the plane $\mathbb R\times[0,1]$. It is clear that we can deform the braid so that its projection has only a finite number of crossing points, each one involving only two strands. We usually draw the crossings as in Figure~\ref{Figure_1} to visualize which strand crosses over and which one crosses under. Moreover, by a further deformation, we can assume that the crossings occur at distinct heights, that is, for distinct values of $t\in [0,1]$. In this way, it is clear that every braid is a product of braids in which only two consecutive strands cross. That is, if one considers for $i=1,\ldots,n-1$ the braids $\sigma_i$ and $\sigma_i^{-1}$ as in Figure~\ref{Figure_2}, it is clear that they are the inverse of each other, and that $\{\sigma_1,\ldots,\sigma_{n-1}\}$ is a set of generators of $B_n$, called the {\it standard generators}, or the {\it Artin generators} of the braid group $B_n$.

\subsection{Braids as automorphisms of the free group}

We shall now give still another interpretation of braids. This is one of the main results in Artin's paper~\cite{Artin}. There is a natural representation of braids on $n$ strands as automorphisms of the free group $F_n$ of rank $n$. Although Artin visualized braids as collections of strands, we believe that it is more natural to define their representation into $Aut(F_n)$ by means of mapping classes, as was done by Magnus~\cite{Magnus_1934}.

\begin{figure}[ht]
\centerline{\includegraphics{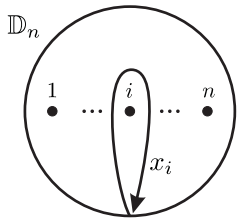}} \caption{The loops $x_1,\ldots,x_n$ are free generators of $\pi_1(\mathbb D_n)$.}
\label{Figure_3}
\end{figure}

We remark that the fundamental group of the $n$-times punctured disc $\mathbb D_n$ is precisely the free group of rank $n$: $\pi_1(\mathbb D_n)=F_n$. If we fix a base point, say in the boundary of $\mathbb D_n$, one can take as free generators the loops $x_1,\ldots,x_n$ depicted in Figure~\ref{Figure_3}. Now a braid $\beta\in B_n$ can be seen as an automorphism of $\mathbb D_n$ up to isotopy, so $\beta$ induces a well defined action on $\pi_1(\mathbb D_n)=F_n$, where a loop $\gamma\in \pi_1(\mathbb D_n)$ is sent to $\beta(\gamma)$. This action is clearly a group homomorphism (respects concatenation of loops), which is bijective as $\beta^{-1}$ yields the inverse action. Hence $\beta$ induces an automorphism of $F_n$, and this gives a representation:
$$
\begin{array}{rccl}    \rho :&  B_n & \longrightarrow & Aut(F_n) \\
                              &  \beta & \longmapsto & \rho_\beta.
\end{array}
$$
The automorphism $\rho_\beta$ can be easily described when $\beta=\sigma_i$, by giving the image of the generators $x_1,\ldots,x_n$ of $F_n$ (see Figure~\ref{Figure_4}). Namely:
$$
  \rho_{\sigma_i}(x_i)= x_{i+1}, \quad \rho_{\sigma_i}(x_{i+1})=x_{i+1}^{-1} x_i x_{i+1}, \quad \rho_{\sigma_i}(x_j)=x_j \mbox{ (if $j\neq i,i+1$).}
$$
The automorphism $\rho_{\sigma_i^{-1}}$ can be easily deduced from $\rho_{\sigma_i}$.  For a general braid $\beta$, written as a product of $\sigma_1,\ldots,\sigma_{n-1}$ and their inverses, the automorphism $\rho_\beta$ is just the composition of the corresponding automorphisms corresponding to each letter.

\begin{figure}[ht]
\centerline{\includegraphics{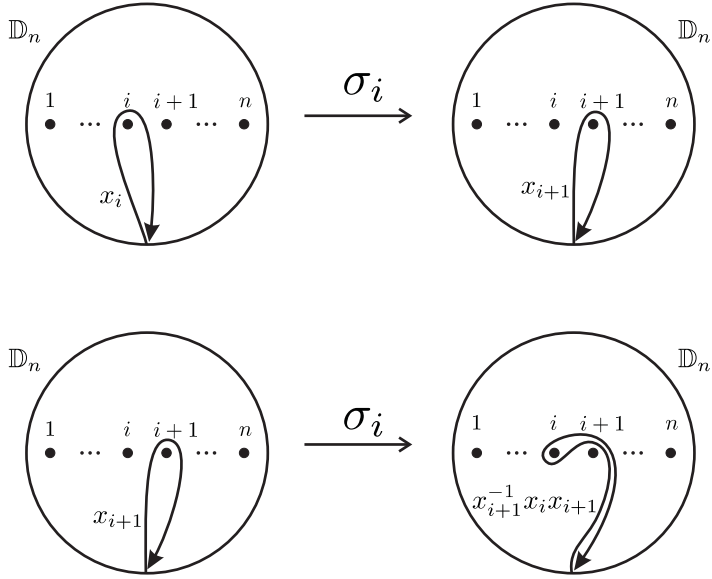}} \caption{Action of $\sigma_i$ on the generators $x_i$ and $x_{i+1}$.}
\label{Figure_4}
\end{figure}

Later we will see that the braid group $B_n$ admits the presentation
\begin{equation}
  B_n= \left\langle \sigma_1,\ldots,\sigma_{n-1} \left| \begin{array}{cl}
    \sigma_i\sigma_j=\sigma_j\sigma_i, & |i-j|>1 \\
    \sigma_i\sigma_j\sigma_i=\sigma_j\sigma_i\sigma_j, & |i-j|=1
  \end{array}\right. \right\rangle.
\end{equation}
It is then very easy to check that $\rho$ is well defined, as $\rho_{\sigma_i\sigma_j}\equiv \rho_{\sigma_j\sigma_i}$ if $|i-j|>1$, and $\rho_{\sigma_i\sigma_j\sigma_i}\equiv \rho_{\sigma_j\sigma_i\sigma_j}$ if $|i-j|=1$. Artin~\cite{Artin2} showed that $\rho$ is faithful by topological arguments, making no use of the above presentation.

Notice that for every $\beta\in B_n$, the automorphism $\rho_\beta$ sends each generator $x_j$ to a conjugate of a generator. Notice also that for each $i=1,\ldots,n-1$, one has $\rho_{\sigma_i}(x_1\cdots x_n)=x_1\cdots x_n$. Hence $\rho_\beta(x_1\cdots x_n)=x_1\cdots x_n$ for every $\beta\in B_n$. This is clear as $x_1\cdots x_n$ corresponds to a loop that runs parallel to the boundary of $\mathbb D_n$, enclosing the $n$ punctures, hence it is not deformed by any braid (up to isotopy). Actually, Artin~\cite{Artin} proved that these two conditions are not only necessary, but also sufficient for an element of $Aut(F_n)$ to be induced by a braid.

\begin{theorem}\cite{Artin} An automorphism $f\in Aut(F_n)$ is equal to $\rho_\beta$ for some $\beta\in B_n$ if and only if the following two conditions are satisfied:
\begin{enumerate}

\item $f(x_i)$ is a conjugate of some $x_j$, for $i=1,\ldots,n$.

\item $f(x_1\cdots x_n)=x_1\cdots x_n$.

\end{enumerate}
\end{theorem}

We recommend the reader to see the beautiful proof of this result in~\cite{Artin} (in German) or in~\cite{Artin2} (in English), based on a simple algebraic argument.

\subsubsection{First solution to the word problem} One important remark, concerning the fact that braids can be seen as automorphisms of the free group $F_n$, is that this immediately yields a solution to the {\it word problem} in $B_n$, as was already noticed by Artin~\cite{Artin}. Given two braids, $\beta_1$ and $\beta_2$, given as words in $\sigma_1,\ldots,\sigma_{n-1}$ and their inverses, one can compute their corresponding automorphisms of $F_n$, $\rho_{\beta_1}$ and $\rho_{\beta_2}$. Then $\beta_1=\beta_2$ if and only if $\rho_{\beta_1}\equiv \rho_{\beta_2}$, and this happens if and only if $\rho_{\beta_1}(x_i)=\rho_{\beta_2}(x_i)\in F_n$ for $i=1,\ldots,n$. Since the word problem in $F_n$ is well known (just need to compute the reduced words associated to $\rho_{\beta_1}(x_i)$ and $\rho_{\beta_2}(x_i)$), this solves the word problem in $B_n$.  We just remark that this algorithm is by no means efficient, and there are several other very efficient algorithms to solve the word problem. But this is historically the first known solution.

\subsubsection{Braids are residually finite and Hopfian}
Let us finish this section with a couple of consequences from the fact that $B_n$ is a subgroup of $Aut(F_n)$, which were noticed by Magnus~\cite{Magnus_1969}. Recall that a group $G$ is {\it residually finite} if the intersection of its finite index subgroups is trivial. Equivalently, $G$ is residually finite if for every nontrivial $a\in G$ there exists a finite group $H$ and a surjection $\varphi:\: G\rightarrow H$ such that $\varphi(a)\neq 1$.

Recall also that a group $G$ is said to be Hopfian if it is not isomorphic to any proper quotient; equivalently, if every surjection from $G$ to itself is an isomorphism.

\begin{theorem}
The braid group $B_n$ is residually finite and Hopfian.
\end{theorem}

\begin{proof}
All the ideas for this result are explained in Section 6.5 of~\cite{Magnus-Karrass-Solitar}, where references to~\cite{Baumslag,Levi,Malcev,Hall} are given.

First notice that if a group is residually finite, so is every subgroup. A very simple argument by Baumslag~\cite{Baumslag}, explained in~\cite{Magnus-Karrass-Solitar}, shows that if a group $G$ is finitely generated and residually finite, then $Aut(G)$ is residually finite. Using these two results, the proof goes as follows:

The group $\mathbb Z^2$ is clearly finitely generated and residually finite, hence $Aut(\mathbb Z^2)$ is residually finite. This group has $F_2$ (the free group of rank two) as a subgroup, for instance $F_2\simeq \left\langle \left(\begin{matrix} 1 & 2 \\ 0 & 1\end{matrix}\right),
\left(\begin{matrix} 1 & 0 \\ 2 & 1 \end{matrix}\right)\right\rangle \subset Aut(\mathbb Z^2)$. Therefore $F_2$ is residually finite. Hence all subgroups of $F_2$ are residually finite, in particular $F_n$ is residually finite for all $n\geq 1$. By Baumslag's argument again, $Aut(F_n)$ is residually finite, and so is $B_n\subset Aut(F_n)$.

On the other hand, Mal'cev~\cite{Malcev} proved that a finitely generated and residually finite group is Hopfian. A short proof can also be found in Section 6.5 of~\cite{Magnus-Karrass-Solitar}.
\end{proof}

\section{Braid groups are torsion free}

One of the best known results of braid groups is that they are torsion free. We shall give several distinct proofs of this fact, and this provides a perfect example of how one can show results in braid groups using completely distinct approaches, each one having interesting consequences, helping to better understand these groups.

\subsection{First proof: Short exact sequences}

Specialists in braid groups use to say: ``Braid groups are torsion free because the configuration space is $K(\pi,1)$". We will try to explain that sentence in this subsection, and we will present the usual way to show that the configuration space $N_n$ is a $K(\pi,1)$ space, given by Fadell and Neuwirth~\cite{Fadell-Neuwirth}.

There are well known short exact sequences of groups, involving braid groups, which were already known by Artin~\cite{Artin}. The first one is quite simple: To each braid in $B_n$ one can associate the permutation it induces on its strands, that is, an element of the symmetric group $\Sigma_n$. This yields a well defined group homomorphism $\eta$ from $B_n$ to $\Sigma_n$. Notice that $\eta(\sigma_i)=(i\ \ i+1)$ for $i=1,\ldots,n-1$. The kernel of $\eta$ is precisely the subgroup of $B_n$ formed by the braids inducing the trivial permutation, that is, the pure braid group $PB_n$. Hence we have an exact sequence:
\begin{equation}\label{E:exact_sequence_PBS}
  1 \rightarrow PB_n \longrightarrow B_n \stackrel{\eta}{\longrightarrow} \Sigma_n \rightarrow 1.
\end{equation}

There is also a natural map that relates pure braid groups of distinct indices. Namely, given a pure braid $\beta\in PB_{n+1}$, one can remove its last strand and will obtain a pure braid $\rho(\beta)\in PB_n$. This yields a well defined homomorphism $\rho: PB_{n+1}\rightarrow PB_n$ which is clearly surjective. The kernel of this map consists of the pure braids in $PB_{n+1}$ whose first $n$ strands form the trivial braid. Up to isotopy, we can consider that the first $n$ strands  are vertical. If we look at this kind of elements as loops in the configuration space $M_{n+1}$, they correspond to a motion of the $n+1$st point, where the points $1,\ldots,n$ do not move. This is of course equivalent to a motion of a point in the $n$-times punctured plane $\mathbb C_n$. In other words, $\ker(\rho)\simeq \pi_1(\mathbb C_n)$. As $\pi_1(\mathbb C_n)$ is isomorphic to $F_n$, the free group of rank $n$, we have the exact sequence:
\begin{equation}\label{E:exact_sequence_FPP}
  1 \rightarrow F_n \stackrel{\iota}{\longrightarrow} PB_{n+1} \stackrel{\rho}{\longrightarrow} PB_n \rightarrow 1.
\end{equation}
In this exact sequence, if $F_n$ is freely generated by $x_1,\ldots,x_n$, we can define
$$
    \iota(x_i)=(\sigma_{n}^{-1} \cdots \sigma_{i+1}^{-1})\: \sigma_i^2 \:(\sigma_{i+1}\cdots \sigma_n),
$$
for $i=1,\ldots,n$.

It is not difficult to see that the exact sequence~(\ref{E:exact_sequence_FPP}) splits. Probably the most elegant way to see this comes from the following result by Fadell and Neuwirth, which implies the existence of the exact sequence, the fact that it is split, and that $M_n$ and $N_n$ are $K(\pi,1)$ spaces:

\begin{theorem}\cite{Fadell-Neuwirth} For $n\geq 1$, the map
$$
\begin{array}{rccc}   p: & M_{n+1} & \longrightarrow & M_n \\
                         & (z_1,\ldots,z_{n+1}) & \longmapsto  & (z_1,\ldots,z_n)
\end{array}
$$
is a locally trivial fiber bundle.
\end{theorem}

One important fact of this fiber bundle is that it admits a cross-section $s:\: M_n \rightarrow M_{n+1}$. This is not explicitly given in~\cite{Fadell-Neuwirth} (it is only given for configuration spaces on punctured manifolds), but it is easy to see~\cite{Paris_2009} that one can take $s((z_1,\ldots,z_n))=(z_1,\ldots,z_n,|z_1|+\cdots+|z_n|+1)$.

Notice that each fiber $p^{-1}\left((z_1,\ldots,z_n)\right)$ is isomorphic to $\mathbb C\backslash\{z_1,\ldots,z_n\}\simeq \mathbb C_n$, so its fundamental group is $F_n$. From the long exact sequence of homotopy groups of this fiber bundle, one has:
$$
\pi_2(\mathbb C_n) \rightarrow \pi_2(M_{n+1}) \rightarrow \pi_2(M_n) \rightarrow \pi_1(\mathbb C_n) \rightarrow \pi_1(M_{n+1})\rightarrow \pi_1(M_n) \rightarrow 1.
$$
It is well known that $\pi_2(\mathbb C_n)=\{1\}$, and we already know the rightmost three groups in the above sequence. Hence we obtain:
\begin{equation}
1 \rightarrow \pi_2(M_{n+1}) \rightarrow \pi_2(M_n) \rightarrow F_n \rightarrow PB_{n+1} \rightarrow PB_n \rightarrow 1,
\end{equation}
for all $n\geq 1$.

Notice that $\pi_2(M_1)=\pi_2(\mathbb C)=\{1\}$. In particular, the above sequence for $n=1$ implies that $\pi_2(M_2)=\{1\}$. Recurrence on $n$ then shows that $\pi_2(M_n)=\{1\}$ for all $n$. This yields the exact sequence~(\ref{E:exact_sequence_FPP}). Also, the cross-section $s: M_n \rightarrow M_{n+1}$ yields a group section  $s:\: PB_n \rightarrow PB_{n+1}$, so the sequence~(\ref{E:exact_sequence_FPP}) splits.

There is an easy way to interpret the group section $s:\: PB_{n} \rightarrow PB_{n+1}$. One can consider the braids in $PB_n$ as collections of strands in $\mathbb D\times [0,1]$, where $\mathbb D \subset \mathbb C$, and the braids in $PB_{n+1}$ as collections of strands in $\mathbb C\times [0,1]$. If one considers a point $z_{n+1}\in \mathbb C\backslash \mathbb D$, then $s$ just adds to a braid in $PB_n$ a single vertical strand based at $z_{n+1}$. This is clearly a group homomorphism which is a section for $\rho$. Notice that this yields the same group section $s$ as the one defined by Fadell-Neuwirth's cross section: Indeed, the cross-section $s$ by Fadell-Neuwirth adds a point which is in the real line, always farther from the origin than the other $n$ points. Hence the braid obtained from that cross-section yields a new strand which does not cross with the other $n$ strands, in the same way as the section we just defined.

Now notice that we can look further to the left on the long exact sequence of homotopy groups, and we get for every $n\geq 1$ and every $k>2$:
$$
   \pi_{k}(\mathbb C_n) \rightarrow \pi_k(M_{n+1}) \rightarrow \pi_k(M_n) \rightarrow \pi_{k-1}(\mathbb C_n),
$$
as $\pi_k(\mathbb C_n)=\{1\}$ for all $k>1$, we have that $\pi_k(M_{n+1})\simeq \pi_k(M_n)$ for all $n\geq 1$. As $\pi_k(M_1)=\pi_k(\mathbb C)=\{1\}$, we finally obtain $\pi_k(M_n)=\{1\}$ for all $n\geq 1$ and all $k> 1$. That is, $M_n$ is a $K(\pi,1)$ space. Therefore, one has:

\begin{theorem}\cite{Fadell-Neuwirth}
$M_n$ and $N_n$ are $K(\pi,1)$ spaces.
\end{theorem}

\begin{proof}
We already showed that $M_n$ is a $K(\pi,1)$ space, as $\pi_k(M_n)=\{1\}$ for all $k>1$. This means that the universal cover of $M_n$ is contractible. As $M_n$ is a covering space of $N_n$, it follows that $N_n$ is also $K(\pi,1)$.
\end{proof}

It is well known that the fundamental group of a $K(\pi,1)$ space is torsion free (see~\cite{Kassel-Turaev} for a proof). Hence we have:

\begin{corollary}
The braid group $B_n$ is torsion free.
\end{corollary}

\subsection{Second proof: Finite order automorphisms}\label{SS:finite_order}

The above proof of the torsion freeness of braid groups, although considered the most natural one by specialists, uses the result that fundamental groups of  $K(\pi,1)$ spaces are torsion free. We will now give another proof, which also uses a couple of well known results, but whose approach to the problem is completely different from the above one.

\begin{theorem}[Nielsen realization theorem for cyclic groups]\cite{Nielsen}
Let $S$ be a closed orientable surface of genus $g\geq 0$, and $q\geq 0$ punctures. Given a mapping class $\varphi \in \mathcal M(S)$ of finite order, there exists a homeomorphism $f\in Aut(S)$ representing $\varphi$, having the same order.
\end{theorem}

\begin{remark}
The actual statement given by Nielsen in~\cite{Nielsen} talks about boundary components instead of punctures, but for him the homeomorphisms that define the mapping classes do not fix the boundary pointwise, but setwise. This implies, for instance, that a Dehn twist along a curve which is parallel to a boundary component determines a trivial mapping class. Hence, a boundary component is equivalent to a puncture in Nielsen's setting. Since we are assuming that our allowed automorphisms fix the boundary pointwise, we prefer to state Nielsen's result the above way. But it is important that Nielsen's original result refers to boundary components, as we will see below.
\end{remark}

We will use Nielsen realization theorem to study braids, by applying it to the $(n+1)$-times punctured sphere $\mathbb S_{n+1}$. Recall that the full twist $\Delta^2\in B_n$ is the braid determined by a Dehn twist along a curve parallel to the boundary of $\mathbb D_n$. Notice that considering the quotient $B_n/\langle\Delta^2\rangle$ corresponds to collapsing the boundary of $\mathbb D_n$ to a puncture, so $B_n/\langle\Delta^2\rangle \simeq \mathcal M(\mathbb S_{n+1},\{\infty\})\subset \mathcal M(\mathbb S_{n+1})$, where we denote by $\{\infty\}$ one of the punctures of $\mathbb S_{n+1}$ (the one corresponding to the boundary of $\mathbb D_n$).

The Nielsen realization theorem tells us that a finite, cyclic subgroup of $\mathcal M(\mathbb S_{n+1})$ can be realized by a subgroup of $Aut(\mathbb S_{n+1})$. Clearly, if the mapping classes that we consider preserve $\{\infty\}$, so do the corresponding automorphisms.  Thanks to this realization, we can deduce results about mapping classes from results about automorphisms. For instance, we can use the following well known result by K\'er\'ekjart\'o~\cite{Kerekjarto} and Eilenberg~\cite{Eilenberg}, beautifully explained in~\cite{Constantin-Kolev}, which states that an orientable, finite order homeomorphism of the disc $\mathbb D^2$ is conjugate to a rotation:

\begin{theorem}\label{T:Kerekjarto}\cite{Kerekjarto,Eilenberg}
Let $f:\: \mathbb D^2 \rightarrow \mathbb D^2$ be a finite order homeomorphism. Then there exists $r\in O(2)$ and a homeomorphism $h:\: \mathbb D^2 \rightarrow \mathbb D^2$ such that $f=hrh^{-1}$.
\end{theorem}

Let us put together the two results above, to show that $B_n$ is torsion free. Suppose that $\alpha\in B_n$ is a braid such that $\alpha^m=1$, and choose an automorphism $g$ of $\mathbb D_n$ representing $\alpha$. If we denote by $[\alpha]$ the canonical projection of $\alpha$ in $B_n/\langle \Delta^2\rangle$, we get that $[\alpha]$ is a mapping class of $\mathbb S_{n+1}$, fixing $\{\infty\}$, such that $[\alpha]^m=1$.
$$
    \begin{array}{ccccc}
            Aut(\mathbb D_n) & \longrightarrow  & \mathcal M(\mathbb D_n) & \longrightarrow & \mathcal M(\mathbb S_{n+1},\{\infty\}) \\
                   g         & \longmapsto      &           \alpha        & \longmapsto     &             [\alpha]
    \end{array}
$$

By the Nielsen realization theorem, there is an automorphism $\phi$ of $\mathbb S_{n+1}$, fixing $\{\infty\}$ and representing $[\alpha]$, such that $\phi^m=\mbox{id}_{\mathbb S_{n+1}}$. Recall that Nielsen's theorem was originally stated for surfaces with boundary components: this allows us to {\it expand} the puncture $\{\infty\}$ to a boundary component, so $\phi$ determines a well defined automorphism $f$ of the punctured disc $\mathbb D_n$. We then have:
$$
    \begin{array}{ccccc}
            Aut(\mathbb D_n) & \longrightarrow  & Aut(\mathbb S_{n+1}, \{\infty\}) & \longrightarrow & \mathcal M(\mathbb S_{n+1},\{\infty\}) \\
                   f         & \longmapsto      &           \phi        & \longmapsto     &             [\alpha]
    \end{array}
$$
where the first map is induced by collapsing the boundary to a puncture, $Aut(\mathbb S_{n+1},\{\infty\})$ denotes the group of automorphisms of $\mathbb S_{n+1}$ fixing the puncture $\{\infty\}$, and the second map is the usual projection.

Notice also that $f$ does not necessarily fix the boundary of $\mathbb D_n$ pointwise. Nevertheless, as $f$ and $g$ represent $[\alpha]$, they are isotopic to each other (where the isotopy fixes the punctures but may rotate the boundary of $\mathbb D_n$). The important fact is that $f^m=\mbox{id}_{\mathbb D_n}$.

We can now use K\'er\'ekjart\'o-Eilenberg's theorem. By filling the punctures, $f$ and $g$ determine automorphisms of $\mathbb D^2$, that we will also denote $f$ and $g$ respectively.  Since $f^m=\mbox{id}_{\mathbb D^2}$, by Theorem~\ref{T:Kerekjarto} it follows that $f$ is conjugate in $Aut(\mathbb D^2)$ to a rotation $r$ of the disc.

Remark that a conjugation in $Aut(\mathbb D^2)$ may change the positions of the $n$ distinguished points (the punctures), and $r$ is a rotation of $\mathbb D^2$ that preserves this set of $n$ distinguished points. One of these points may be the center of $\mathbb D^2$, and in this case it is a fixed point of $r$. All other distinguished points have orbits of the same size (depending on the angle of the rotation), each orbit consisting of a set of points evenly distributed in a circumference centered at the origin. By a further conjugation in $Aut(\mathbb D^2)$, we can place all orbits of non-fixed distinguished points into the same circumference centered at the origin, so that all distinguished points (except the fixed one if this is the case) are evenly distributed in that circumference (see Figure~\ref{Figure_5}).

\begin{figure}[ht]
\centerline{\includegraphics{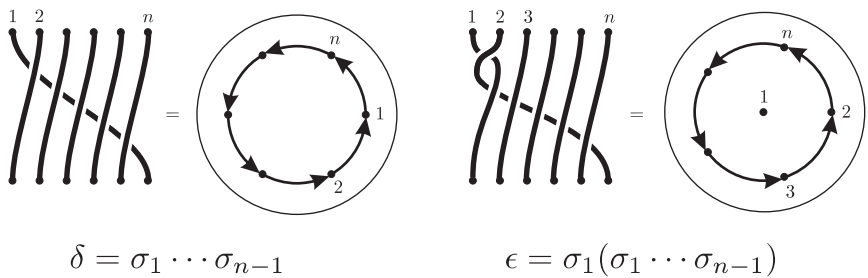}} \caption{Every periodic braid is conjugate to a power of either $\delta$ or $\epsilon$.}
\label{Figure_5}
\end{figure}

By performing an isotopy close to the boundary (which becomes trivial after collapsing the boundary to a puncture), we can transform the conjugate automorphisms $f$ and $r$ into conjugate automorphisms $f'$ and $r'$ which fix the boundary pointwise (but which are no longer of finite order). Looking at the motion of the $n$ distinguished points performed by $r'$ (a rotation), it is clear that $r'$ is conjugate in $Aut(\mathbb D^2)$ to an automorphism $h$ fixing the set $\{1,\ldots,n\}$ and
inducing one of the following braids: either a power of $\epsilon=\sigma_1(\sigma_1\cdots \sigma_{n-1})$ or a power of $\delta=\sigma_1\cdots\sigma_{n-1}$ (depending whether the center of $\mathbb D^2$ is a distinguished point or not, see Figure~\ref{Figure_5}).

Recall that we started with a braid $\alpha$, and a representative $g\in Aut(\mathbb D^2)$. Up to isotopy (possibly rotating the boundary), we have seen that $g$ is conjugate in $Aut(\mathbb D^2)$ to $h$. As both $g$ and $h$ preserve $\{1,\ldots,n\}$ setwise, and fix the boundary pointwise, the conjugating automorphism preserves $\{1,\ldots,n\}$ setwise, and can be taken to fix the boundary pointwise (recall that $g$ is considered up to rotations of the boundary). This means that the braids induced by $g$ and $h$ are conjugate, up to multiplication by a power of $\Delta^2$ (a full rotation of the boundary). In other words, $\alpha$ is conjugate to either $\delta^t \Delta^{2m}$ or to $\epsilon^t\Delta^{2m}$ for some $t,m\in \mathbb Z$. As $\Delta^2=\delta^{n}=\epsilon^{n-1}$, we finally obtain that $\alpha$ is conjugate to either $\delta^k$ or $\epsilon^k$ for some $k\in \mathbb Z$.

It just remains to show that $k=0$. This can be easily seen by noticing that $\delta$ and $\epsilon$ do not have finite order (while $\alpha$ does). Consider the homomorphism $s:\: B_n\rightarrow \mathbb Z$ that sends $\sigma_i$ to $1$, for $i=1,\ldots,n-1$. As the defining relations of $B_n$ are homogeneous, $s$ is well defined. As $s(\delta^k)=(n-1)k$ and $s(\epsilon^k)=nk$ for every every integer $k$, it follows that neither $\delta$ nor $\epsilon$ has finite order. This shows that $B_n$ is torsion free.

\subsection{Third proof: Left orderability}

We will now give yet another proof of the torsion-freeness of braid groups. The first one was mainly topological, while the second one was more geometric; The third one will be an algebraic proof, relying on the fact that braid groups are left-orderable.

\begin{definition}
A group $G$ is said to be left-orderable if it admits a total order of its elements which is invariant under left-multiplication. That is,
$$
    a < b \Rightarrow ca < cb, \qquad \forall a,b,c\in G.
$$
\end{definition}

Notice that the total order $\leqslant$ defines a semigroup contained in $G$, namely $P=\{a\in G; \; 1<a\}$, called the semigroup of {\it positive elements}. As the order is invariant under left multiplication, one has $1<a$ if and only if $a^{-1}<1$. This means that the set $\{b\in G; \; b<1\}$ is precisely $P^{-1}$, the set of inverses of elements in $P$. Notice that one then has $G=P\sqcup \{1\}\sqcup P^{-1}$.  This is actually an equivalent condition to left-orderability: A group $G$ is left-orderable if and only if it admits a sub-semigroup $P$ such that $G=P\sqcup \{1\}\sqcup P^{-1}$. We just need to define the order by saying that $a<b$ if and only if $a^{-1}b\in P$.

The relation between braid groups and orderability is given by the following result:

\begin{theorem}\cite{Dehornoy_1994}
The braid group $B_n$ is left-orderable for all $n\geq 1$.
\end{theorem}

The left-order defined in~\cite{Dehornoy_1994} is called {\it Dehornoy's ordering}.  The sub-semigroup of $B_n$ which defines Dehornoy's ordering is defined as follows: $P$ is the set of nontrivial braids that can be expressed as a word in which the $\sigma_i$ of smallest index appears only with positive exponents. It is clear that $P$ is a semigroup; but proving that $B_n=P\sqcup \{1\}\sqcup P^{-1}$ using Dehornoy's approach is highly nontrivial.

But there is an interpretation of Dehornoy's ordering, given in~\cite{FGRRW}, which shows in an easy way that $B_n=P\sqcup \{1\}\sqcup P^{-1}$. One just needs to look at braids as (isotopy classes of) automorphisms of $\mathbb D_n$. Recall that we are considering $\mathbb D_n$ as a subset of $\mathbb C$, and the $n$ punctures lying in the real line. Let $E=\mathbb R\cap \mathbb D_n$, that is, the diameter containing the punctures. Notice that $E$ is the disjoint union of $n+1$ segments, that we will denote $E_1,\cdots,E_{n+1}$. Given a braid $\beta$, we consider the image of the above diameter, $\beta(E)$. It is defined up to isotopy of $\mathbb D_n$, and it can always be isotoped to a curve satisfying the following condition: whenever it is possible, $\beta(E_i)$ will be isotoped to a horizontal segment (i.e. some $E_j$), and if this is not the case, $\beta(E_i)$ will be isotoped to have the minimal possible number of intersections with $E$. Given such a representation of $\beta(E)$, one can say that $\beta\in P$ if the initial part of the segment $\beta(E_i)$ belongs to the lower half-disc, where $i$ is the smallest index such that $\beta(E_i) \neq E_i$. In other words, $\beta\in P$ if the first non-horizontal segment in $\beta(E)$ goes downwards. An example can be seen in Figure~\ref{Figure_6}: The braid $\beta=\sigma_2^{-1}\sigma_3\sigma_2\in B_4$ sends $E$ to the curve drawn in the picture, so $\beta$ is a positive braid. Algebraically, $\beta$ is positive as it is equal to $\sigma_3\sigma_2\sigma_3^{-1}$, and in this word only positive powers of $\sigma_2$ appear.

\begin{figure}[ht]
\centerline{\includegraphics{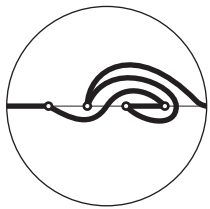}} \caption{The braid $\beta=\sigma_2^{-1}\sigma_3\sigma_2\in B_4$ is positive.}
\label{Figure_6}
\end{figure}

\begin{theorem}\cite{FGRRW}
The set $P$ defined above is a well defined sub-semigroup of $B_n$, such that $B_n=P\sqcup\{1\}\sqcup P^{-1}$. Moreover, $P$ coincides with the semigroup defining Dehornoy's ordering.
\end{theorem}

The more delicate part of the above theorem's proof is to rigourously show that $P$ is well defined, but it is in any case quite manifest that this is true, and one can find the details in~\cite{FGRRW}. Notice that it is trivial from the definition of $P$ that $B_n=P\sqcup\{1\}\sqcup P^{-1}$, and it is not difficult to show that $P$ is a semigroup. Once Denhornoy showed that every nontrivial braid belongs to either $P$ or $P^{-1}$ (with Dehornoy's definition), is is clear that a positive element in Dehornoy's ordering is also positive in the sense given by the above theorem, so both orderings coincide.

We point out that there is a recent work by Bacardit and Dicks~\cite{Bacardit-Dicks}, giving an alternative proof of Dehornoy's result, using only algebraic arguments.

Once we know that $B_n$ is left-orderable, it follows immediately that $B_n$ is torsion-free, by a classical result.

\begin{theorem}
A left-orderable group is torsion-free.
\end{theorem}

\begin{proof}
Suppose that $G$ is left-orderable, and let $\beta\in G$ with $\beta\neq 1$. Suppose $1<\beta$. Multiplying from the left by $\beta$ one has $\beta<\beta^2$. Multiplying again, one gets $\beta^2<\beta^3$. Iterating this process one obtains a chain $1<\beta<\beta^2<\beta^3<\cdots $  As $<$ is an order, it is not possible that $\beta^m=1$ for some $m$.  The case $\beta<1$ is analogous.
\end{proof}

\begin{corollary}
$B_n$ is torsion-free.
\end{corollary}

We just point out, to end this section, that the pure braid group $P_n$ is bi-orderable. This means that there is a total order of its elements which is invariant by left and also right multiplication. Some good references about orderability and braid groups are~\cite{DDRW1,DDRW2}.

\section{Presentations of the braid groups}

One of the best known features of the braid groups is the finite presentation discovered by Artin in~\cite{Artin}. We have already mentioned the generators $\sigma_1,\ldots,\sigma_{n-1}\in B_n$. The complete presentation is as follows:
\begin{equation}\label{E:presentation}
B_n=\left\langle \sigma_1,\ldots,\sigma_{n-1} \left|\begin{array}{cl} \sigma_i\sigma_j=\sigma_j\sigma_i, &  |i-j|>1 \\
       \sigma_i\sigma_j\sigma_i=\sigma_j\sigma_i\sigma_j, & |i-j|=1\end{array}  \right.  \right\rangle
\end{equation}

The proof by Artin of the completeness of this presentation is more an indication than a rigorous proof. There have been other proofs by Magnus~\cite{Magnus_1934}, Bohnenblust~\cite{Bohnenblust}, Chow~\cite{Chow}, Fadell and Van Buskirk~\cite{Fadell-VanBuskirk}, or Fox and Neuwirth~\cite{Fox-Neuwirth}. See also~\cite{Birman}. Most of these proofs use the short exact sequences (\ref{E:exact_sequence_PBS}) and (\ref{E:exact_sequence_FPP}).

Basically, one can use the fact that $PB_2\simeq \mathbb Z$, and use Reidemeister-Schreier method applied to (\ref{E:exact_sequence_FPP}) to construct a presentation of $PB_n$, by induction on $n$. Then one can use (\ref{E:exact_sequence_PBS}) to deduce that the presentation (\ref{E:presentation}) is correct.

These proofs are quite technical, involving lots of calculations and sometimes containing small mistakes. We will give here a couple of proofs that, in our opinion, are not so technical, and hopefully will make the reader believe in the correctness of the presentation.

\subsection{First proof: Braid combing}

The first proof is based on an argument by Zariski~\cite{Zariski}. Although Zariski applied this argument to the braid group of the sphere, it can easily be applied to classical braids.

\begin{proposition}
The presentation {\rm (\ref{E:presentation})} is correct.
\end{proposition}

\begin{proof}
Let $W=\sigma_{i_1}^{e_1}\cdots \sigma_{i_m}^{e_m}$ be a word in $\sigma_1,\ldots,\sigma_{n-1}$ and their inverses, and suppose that the braid determined by $W$ is trivial. We must show that one can obtain the trivial word starting with $W$, and applying only the relations in presentation (\ref{E:presentation}), together with insertions or deletions of subwords of the form $(\sigma_i^{\pm 1} \sigma_i^{\mp 1})$.

For every $k=0,\ldots, m$, let $j_k$ be the position of the $n$-th puncture at the end of the motion represented by $\sigma_{i_1}^{e_1}\cdots \sigma_{i_k}^{e_k}$. As $W$ represents the trivial braid, it is clear that $j_0=j_m=n$. Now denote $\alpha_i=\sigma_i \sigma_{i+1}\cdots \sigma_{n-1}$, for $i=1,\ldots,n-1$, and also denote $\alpha_n=1$. Then $\alpha_i$ represents a braid that sends the $i$-th puncture to the $n$-th position.

It is clear that, using only permitted insertions, we can transform our word $W$ into:
$$
   W\equiv (\alpha_{j_0}^{-1} \sigma_{i_1}^{e_1} \alpha_{j_1})(\alpha_{j_1}^{-1} \sigma_{i_2}^{e_2} \alpha_{j_2})\cdots (\alpha_{j_{m-1}}^{-1} \sigma_{i_m}^{e_m} \alpha_{j_m}).
$$
This holds as $\alpha_{j_0}=\alpha_{j_m}=1$. Now each parenthesized factor has one of the following forms:
\begin{enumerate}

 \item $(\sigma_{n-1}^{-1}\cdots \sigma_{i}^{-1}) \sigma_i (\sigma_{i+1}\cdots \sigma_{n-1})$. This is clearly equivalent to the trivial word, and can be removed.

 \item $(\sigma_{n-1}^{-1}\cdots \sigma_{i}^{-1}) \sigma_i^{-1} (\sigma_{i+1}\cdots \sigma_{n-1})$. We will denote this word $x_i^{-1}$.

 \item $(\sigma_{n-1}^{-1}\cdots \sigma_{i}^{-1}) \sigma_{i-1} (\sigma_{i-1}\cdots \sigma_{n-1})$. We will denote this word $x_{i-1}$.

 \item $(\sigma_{n-1}^{-1}\cdots \sigma_{i}^{-1}) \sigma_{i-1}^{-1} (\sigma_{i-1}\cdots \sigma_{n-1})$. This is equivalent to the trivial word, and can be removed.

 \item $(\sigma_{n-1}^{-1}\cdots \sigma_{i}^{-1}) \sigma_k^{\pm 1} (\sigma_{i}\cdots \sigma_{n-1})$ with $k<i-1$. In this case, $\sigma_k^{\pm 1}$ commutes with the other letters, so we can use permitted relations to replace this word by the letter $\sigma_k^{\pm 1}$.

 \item $(\sigma_{n-1}^{-1}\cdots \sigma_{i}^{-1}) \sigma_k^{\pm 1} (\sigma_{i}\cdots \sigma_{n-1})$  with  $k>i$.  It is very easy to see that, if $k>i$, using the braid relations one has
 \begin{equation}\label{E:basic_relation}
        \sigma_k (\sigma_i\cdots \sigma_{n-1}) \equiv (\sigma_i \cdots \sigma_{n-1})\sigma_{k-1}.
 \end{equation}
    Therefore the above word is equivalent to $\sigma_{k-1}^{\pm 1}$.

\end{enumerate}

Notice that, by the above procedure, we have replaced our original word $W$ by a word in $\sigma_1,\cdots, \sigma_{n-2},x_1,\cdots,x_{n-1}$ and their inverses. It is important that $\sigma_{n-1}$ and $\sigma_{n-1}^{-1}$ never appear in this writing alone, but always as parts of some word $x_i^{\pm 1}$.

Now, for $i=1,\ldots,n-2$ and $j=1,\ldots,n-1$, the word $\sigma_i^{-1} x_j \sigma_i$ can be written as a product of $x_1,\ldots,x_{n-1}$ and their inverses, by using only permitted relations. Indeed, if $i<j-1$ one can slide $\sigma_i$ to the left and the resulting word is $x_j$. If $i=j-1$, that is $j=i+1$, one has the following:
\begin{eqnarray*}
   \sigma_i^{-1} x_{i+1} \sigma_i & \equiv & \fbox{$\sigma_i^{-1} (\sigma_{n-1}^{-1}\cdots \sigma_{i+2}^{-1})$} \sigma_{i+1}^2 \fbox{$(\sigma_{i+2}\cdots \sigma_{n-1})\sigma_i$}
\\ & \equiv &  (\sigma_{n-1}^{-1}\cdots \sigma_{i+2}^{-1}) \sigma_i^{-1} \sigma_{i+1}\quad \sigma_{i+1} \sigma_{i} (\sigma_{i+2}\cdots \sigma_{n-1})
\\ & \equiv &  (\sigma_{n-1}^{-1}\cdots \sigma_{i+2}^{-1})\fbox{$ \sigma_i^{-1} \sigma_{i+1}\sigma_i$}\: \fbox{$\sigma_i^{-1} \sigma_{i+1}\sigma_{i}$} (\sigma_{i+2}\cdots \sigma_{n-1})
\\ & \equiv &  (\sigma_{n-1}^{-1}\cdots \sigma_{i+2}^{-1})\sigma_{i+1} \sigma_{i}\fbox{$\sigma_{i+1}^{-1} \sigma_{i+1}$} \sigma_{i}\sigma_{i+1}^{-1} (\sigma_{i+2}\cdots \sigma_{n-1})\\
& \equiv & x_{i+1}\:x_i \:x_{i+1}^{-1},
\end{eqnarray*}
where the last equality is obtained just by permitted insertions.

If $i=j$ one has
\begin{eqnarray*}
   \sigma_i^{-1} x_i \sigma_i & \equiv & \sigma_i^{-1} (\sigma_{n-1}^{-1}\cdots \sigma_{i+2}^{-1}\sigma_{i+1}^{-1}) \sigma_i^2 (\sigma_{i+1}\fbox{$\sigma_{i+2}\cdots \sigma_{n-1})\sigma_i$}
\\ & \equiv & \sigma_i^{-1} (\sigma_{n-1}^{-1}\cdots \sigma_{i+2}^{-1}\sigma_{i+1}^{-1}) \fbox{$\sigma_i^2 \sigma_{i+1}\sigma_i$}(\sigma_{i+2}\cdots \sigma_{n-1})
\\
& \equiv &\sigma_i^{-1} (\sigma_{n-1}^{-1}\cdots \sigma_{i+2}^{-1} \fbox{$\sigma_{i+1}^{-1}) \sigma_{i+1}$} \sigma_{i}\sigma_{i+1}^2(\sigma_{i+2}\cdots \sigma_{n-1})
\\
& \equiv &\fbox{$\sigma_i^{-1} (\sigma_{n-1}^{-1}\cdots \sigma_{i+2}^{-1}) \sigma_{i}$}\sigma_{i+1}^2(\sigma_{i+2}\cdots \sigma_{n-1})
\\
& \equiv & (\sigma_{n-1}^{-1}\cdots \sigma_{i+2}^{-1}) \sigma_{i+1}^2(\sigma_{i+2}\cdots \sigma_{n-1})
\\
& \equiv & x_{i+1}.
\end{eqnarray*}

Finally, if $i>j$, one has $\sigma_i x_j \equiv x_j \sigma_i$, as one can see by sliding $\sigma_i$ to the right, using the obvious relation at each time. Hence if $i>j$ one has $\sigma_i^{-1} x_j \sigma_i \equiv x_j$.

It is clear that the above equations also imply that $\sigma_i x_j \sigma_i^{-1}$ can be written as a word in $x_1,\ldots,x_{n-1}$ and their inverses. The resulting word is $x_j$ if either $i<j-2$ or $i>j$, it is  $x_{i}$ if $i=j-1$, and it is $x_i^{-1}x_{i+1}x_i$ if $i=j$.

Therefore, starting with the word $W$, once we have rewritten it as a word in $\sigma_1,\ldots,\sigma_{n-2}, x_1,\ldots,x_{n-1}$ and their inverses, we can collect all the $\sigma_{i}^{\pm 1}$ on the right, so that we can write:
$$
    W \equiv W_1\; W_2,
$$
where $W_1$ is a word in $x_1,\ldots,x_{n-1}$ and their inverses, and $W_2$ is a word in $\sigma_{1}\cdots \sigma_{n-2}$ and their inverses.

Finally, we just need to notice that by the split exact sequence (\ref{E:exact_sequence_FPP}), one has $PB_n=F_{n-1}\rtimes PB_{n-1}$, so every pure braid can be decomposed in a unique way as a product of a braid in $\iota(F_{n-1})$ and a braid in $PB_{n-1}$ (with the usual inclusion of $PB_{n-1}$ into $PB_n$). We remark that $\iota(F_{n-1})$ is the free subgroup of $PB_n$ freely generated by $x_1,\ldots,x_{n-1}$. Hence, as $W$ is pure ($W$ represents the trivial braid), the decomposition $W_1W_2$ is unique, meaning that $W_1$ represents the trivial element in $F_{n-1}$ and $W_2$ represents the trivial element in $PB_{n-1}$. As $x_1,\ldots,x_n$ is a free set of generators of $F_{n-1}$, it follows that $W_1$ can be reduced to the trivial word by a sequence of permitted deletions. Therefore $W\equiv W_2$, which is a word in $\sigma_1,\ldots,\sigma_{n-2}$ and their inverses representing the trivial braid in $B_{n-1}$. The result then follows by induction on $n$.
\end{proof}

\subsection{Second proof: Fundamental groups of cell complexes}

The above proof, although elementary, still involves some technical calculations, and does not help to see why presentation (\ref{E:presentation}) is a natural one. We present now the beautiful proof by Fox and Neuwirth~\cite{Fox-Neuwirth}, in which braid groups are seen as fundamental groups of cell complexes.

It is well known that given a regular cell complex $\mathcal C$ of dimension $m$, and a subcomplex $\mathcal C'$ of dimension $m-2$, one can compute a presentation of the fundamental group $\pi_1(\mathcal C\backslash \mathcal C')$ in the following way:
\begin{enumerate}

\item Consider the dual graph $\Gamma$ in $\mathcal C\backslash \mathcal C'$, that is, a graph having a vertex for each $m$-cell of $\mathcal C\backslash \mathcal C'$, and an edge connecting two vertices of $\Gamma$ for every $(m-1)$-cell adjacent to the corresponding $m$-cells.

\item Choose a maximal tree $T$ in $\Gamma$.

\item There is a generator of $\pi_1(\mathcal C\backslash \mathcal C')$ for every edge $e\in \Gamma\backslash T$ (corresponding to a $(m-1)$-cell).

\item There is a relation for each $(m-2)$-cell in $\mathcal C\backslash \mathcal C'$.

\end{enumerate}

The generators mentioned in step 3 can be constructed as follows: Fix as a base point a vertex $v_0$ of $\Gamma$. Given an edge $e\in \Gamma\backslash T$, its corresponding generator is a loop that goes, along $T$, from $v_0$ to the one of the  endpoints of $e$, then moves along $e$, and finally goes back to $v_0$ along $T$.

The relations in step 4 are defined as follows: Given a $(m-2)$-cell $c$, we can consider the $(m-1)$-cells adjacent to $c$, which are positioned in a well defined cyclic order (cutting these cells by a transverse plane, the situation looks like a vertex with adjacent edges). The collection of these $(m-1)$-cells, determines a loop in $\Gamma$ (up to orientation). At least one of the edges in this loop does not belong to $T$, as $T$ is a tree, and in this case it corresponds to a generator. Reading these edges (generators) in the corresponding order, with the corresponding orientation, provides the relation in $\pi_1(\mathcal C\backslash \mathcal C')$.

We know that the braid group $B_n$ is the fundamental group of the configuration space $N_n$. The brilliant idea by Fox and Neuwirth was to construct cell complexes $\mathcal C$ and $\mathcal C'$ such that $\mathcal C\backslash \mathcal C'=N_n$, in such a way that the above procedure yields the well known presentation of $B_n$ given in (\ref{E:presentation}). Furthermore, the cell comlexes they define are quite easy to understand.

Consider $\mathbb C^n/\Sigma_n$, where the symmetric group $\Sigma_n$ acts on $\mathbb C_n$ by permuting coordinates. The space $\mathbb C^n/\Sigma_n$ has real dimension $2n$.

Notice that if we consider a complex number as a pair of real numbers, then $\mathbb C$ can be totally ordered, using the lexicographical order. More precisely, given $z_1=a_1+b_1i$ and $z_2=a_2+b_2i$, we say that $z_1\leq_{lex}z_2$ if and only if either $a_1<a_2$ (we will say that $z_1<z_2$), or $a_1=a_2$ and $b_1< b_2$ (we will say that $z_1\veebar z_2$), or $z_1=z_2$.   Geometrically, $z_1 <z_2$ means that $z_1$ is {\it to the left} of $z_2$, and $z_1\veebar z_2$ means that $z_1$ is {\it below} $z_2$.

Notice that every point in $\mathbb C^n/\Sigma_n$ is a family of $n$ undistinguishable (not necessarily distinct) points $\{z_1,\ldots,z_n\}\subset \mathbb C$. These points can be ordered lexicographically as above, so we can write $z_1 \square z_2\square \cdots \square z_n$, where $\square$ can be either $<$ or $\veebar$ or $=$.

We can now define a symbol $\theta =(z_1\square z_2 \square \cdots \square z_n)$ in the above way. Each of these symbols defines a cell:
$$
     C_{\theta}=\{\{z_1,\ldots,z_n\}\in \mathbb C^n/\Sigma_n\mbox{ such that } \theta\}.
$$
For instance, if $\theta=(z_1\veebar z_2 <z_3)$, then $C_{\theta}$ corresponds to the configurations of three points in $\mathbb C$ such that two of them are one below the other, and the third one is more to the right.

It is not hard to see that the sets $C_{\theta}$ determine a regular cell decomposition $\mathcal C$ of $\mathbb C^m/\Sigma_n$. Moreover, those $C_{\theta}$ for which $\theta$ involves at least one equality, determine a regular cell decomposition of the big diagonal $\mathcal D/\Sigma_n$. It is clearly a subcomplex of the above one, and moreover $(\mathbb C^n/\Sigma_n)\backslash (\mathcal D/\Sigma_n) = N_n$.  Therefore, we have $N_n$ considered as the complement of a subcomplex of a regular cell complex. The above procedure will then yield a presentation of $\pi_1(N_n)$, that is, a presentation of $B_n$.

Clearly, the real dimension of a cell $C_{\theta}$ is
$$
   \dim_{\mathbb R}(C_{\theta})= 2n - (\mbox{no. of } \veebar) - 2(\mbox{no. of }=).
$$
Hence $\mathcal C$ has dimension $2n$.  Notice that there is only one $(2n)$-cell, namely $C_{\theta}$ for $\theta=(z_1<z_2<\cdots<z_n)$. Hence $\Gamma$ has only one vertex, $v_0$, and then $T=\{v_0\}$. We can consider $v_0=\{1,\ldots,n\}\subset \mathbb C$, the usual base points for braids.

Let us compute the generators. As we saw above, they correspond to the $(2n-1)$-cells, that is, to $C_{\theta_i}$ with $\theta_i=(z_1<\cdots <z_i \veebar z_{i+1}<\cdots <z_n)$, for $i=1,\ldots,n-1$. We are then going to obtain $n-1$ generators. To construct the $i$-th generator, we start at $v_0$,  we must cross the cell $C_{\theta_i}$ and go back to $v_0$ again. This corresponds to moving the points $i$ and $i+1$ so that they swap positions, the $i$-th one passing below the $(i+1)$-st. That is, this corresponds to the motion $\sigma_i$ (see Figure~\ref{Figure_7}). Therefore, the generators we obtain are precisely $\sigma_1,\ldots,\sigma_{n-1}$.

\begin{figure}[ht]
\centerline{\includegraphics{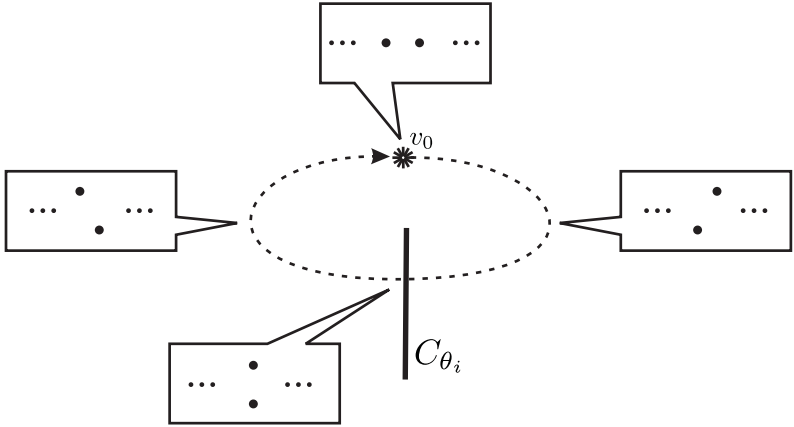}} \caption{This represents a loop based at the point $v_0=(1,\ldots,n)$, moving along the $2n$-cell $C_\theta$ and crossing the $(2n-1)$-cell $C_{\theta_i}$. The motion of the punctures corresponds to the generator $\sigma_i$.}
\label{Figure_7}
\end{figure}

The relations are given by the cells of dimension $2n-2$. we will distinguish two cases: Those corresponding to
$$
 \theta_{i,j}=(z_1<\cdots <z_i\veebar z_{i+1} <\cdots < z_j \veebar z_{j+1} <\cdots < z_n)
$$ with $i+1<j$, and those corresponding to
$$
 \theta_{i,i+1} = (z_1<\cdots <z_i\veebar z_{i+1} \veebar z_{i+2} <\cdots < z_n).
$$

In the first case, $C_{\theta_{i,j}}$ corresponds to a configuration in which there are two pairs of points, $z_i$ below $z_{i+1}$, and also $z_j$ below $z_{j+1}$, while the others lay in distinct vertical lines. This $(2n-2)-cell$ is adjacent to two $(2n-1)$-cells, namely $C_{\theta_i}$ and $C_{\theta_j}$, in the way shown in Figure~\ref{Figure_8}. In order to do a loop around $C_{\theta_{i,j}}$, one must cross each adjacent cell twice, once in each sense, so the corresponding relation reads $\sigma_i \sigma_{j}\sigma_i^{-1}\sigma_j^{-1}=1$, or equivalently:
$
 \sigma_i\sigma_j=\sigma_j\sigma_i.
$

\begin{figure}[ht]
\centerline{\includegraphics{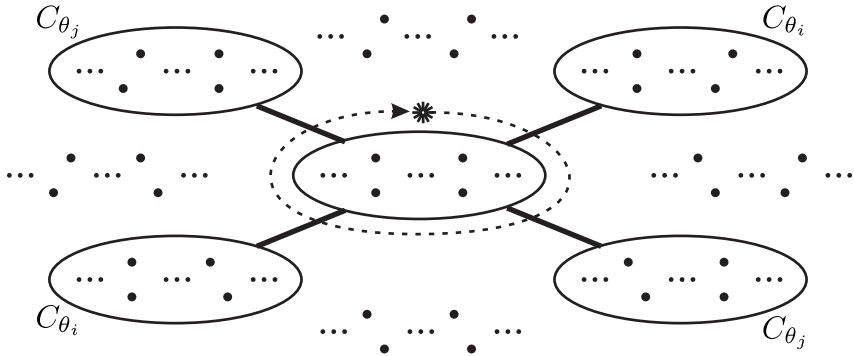}} \caption{This represents a loop around the $(2n-2)$-cell $C_{\theta_{i,j}}$. It crosses $C_{\theta_i}$,  $C_{\theta_j}$,  then $C_{\theta_i}$ in the opposite sense, and finally $C_{\theta_j}$ in the opposite sense. Looking at the movement of the punctures involved, we see that this corresponds to the relation $\sigma_i\sigma_j\sigma_i^{-1}\sigma_j^{-1}=1$. }
\label{Figure_8}
\end{figure}

The second case corresponds to a $(2n-2)$-cell, $C_{\theta_{i,i+1}}$ in which there are three points in the same vertical line, $z_i$ is below $z_{i+1}$, which is below $z_{i+2}$. This cell is also adjacent to two $(2n-1)$-cells, depending whether one of the three points {\it escapes} from the vertical line to the left, or to the right. That is, $C_{\theta_{i,i+1}}$ is adjacent to $C_{\theta_i}$ and $C_{\theta_{i+1}}$. This time the adjacency is as shown in Figure~\ref{Figure_9}, so in order to do a loop around $C_{\theta_{i,i+1}}$, one crosses each $(2n-1)$-cell three times, with the corresponding orientations, yielding a relation: $\sigma_i \sigma_{i+1}\sigma_i \sigma_{i+1}^{-1} \sigma_i^{-1} \sigma_{i+1}^{-1}=1$, or in other words:
$$
   \sigma_i\sigma_{i+1}\sigma_i = \sigma_{i+1}\sigma_i \sigma_{i+1}.
$$
Therefore, this argument by Fox and Neuwirth, shows not only the correctness of (\ref{E:presentation}), but also why it is a natural presentation of $B_n$.

\begin{figure}[ht]
\centerline{\includegraphics{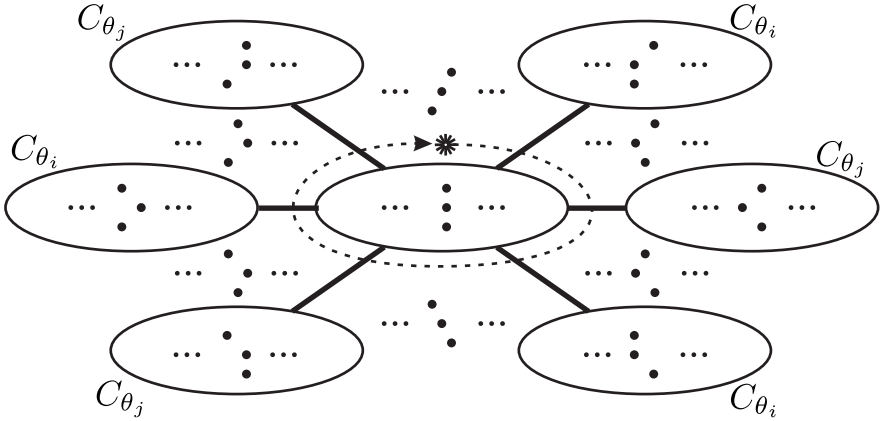}} \caption{This represents a loop around the $(2n-2)$-cell $C_{\theta_{i,i+1}}$. It crosses the cell $C_{\theta_i}$, then $C_{\theta_{i+1}}$, then $C_{\theta_i}$ again, and then it crosses, in the opposite sense, $C_{\theta_{i+1}}$, $C_{\theta_i}$ and $C_{\theta_{i+1}}$. Looking at the movement of the punctures involved, we see that this corresponds to the relation $\sigma_i\sigma_{i+1}\sigma_i\sigma_{i+1}^{-1}\sigma_{i}^{-1}\sigma_{i+1}^{-1}=1$. }
\label{Figure_9}
\end{figure}

\section{Garside structure}

Let us now present one of the most fruitful algebraic features of braid groups. The story begins with the work of Garside~\cite{Garside}, who discovered some properties of braid groups which allowed him to solve the word problem (in a new way) and the conjugacy problem (for the first time) in braid groups. He also proved in a new, simple way that $\langle \Delta^2\rangle$ is the center of $B_n$.

Garside's results were generalized to all Artin-Tits groups of spherical type by Brieskorn and Saito~\cite{Brieskorn-Saito}. Later, Dehornoy and Paris~\cite{Dehornoy-Paris} introduced {\it Garside groups} as, basically, the class of groups satisfying the algebraic properties discovered by Garside. Garside groups include, of course, all Artin-Tits groups of spherical type, in particular braid groups. This implies that the properties of braid groups that one is able to show using the techniques introduced by Garside (and developed by several other authors), will hold in every Garside group. Some authors like David Bessis~\cite{Bessis}, Fran\c{c}ois Digne and Jean Michel~\cite{Digne-Michel}, Daan Krammer~\cite{Krammer3} and Patrick Dehornoy~\cite{Dehornoy_2009}, are extending the scope by introducing {\it Garside categories}, for which Garside groups are a particular case.

The main ideas in Garside's work are the following. First notice that the relations in  presentation (\ref{E:presentation}) involve only positive powers of the generators, hence one can consider the {\it monoid} $B_n^+$ determined by that presentation. Elements of $B_n^+$ are words in $\sigma_1,\ldots,\sigma_{n-1}$ (but not their inverses), and two such words are equivalent if and only if one can obtain one from the other by iteratively replacing subwords of the form $\sigma_i\sigma_j$ ($|i-j|>1$), respectively $\sigma_i\sigma_j\sigma_i$ ($|i-j|=1$), by $\sigma_j\sigma_i$, respectively $\sigma_j\sigma_i\sigma_j$.

In the monoid $B_n^+$, there is a natural partial order. Namely, given $a,b\in B_n^+$ we say that $a\preccurlyeq b$ if $ac=b$ for some $c\in B_n^+$. We say that $a$ is a {\it prefix} of $b$. Garside does not mention this order in~\cite{Garside}, but it will be clearer to explain his results in this way. Notice that $\preccurlyeq$ is a partial order which is invariant under left multiplication, that is, $a\preccurlyeq b$ implies $xa\preccurlyeq xb$ for every $a,b,x\in B_n^+$.

Given such a partial order, one may wonder whether there exist unique greatest common divisors and least common multiples with respect to $\preccurlyeq$. That is, given $a,b\in B_n^+$, does it exist a unique $d\in B_n^+$ such that $d\preccurlyeq a$, $d\preccurlyeq b$ and $d'\preccurlyeq d$ for every $d'$ being a common prefix of $a$ and $b$? And does there exist a unique $m\in B_n^+$ such that $a\preccurlyeq m$, $b\preccurlyeq m$ and $m\preccurlyeq m'$ for every $m'$ having $a$ and $b$ as prefixes?  In such cases we will write $d=a\wedge b$ and $m=a\vee b$. Notice that we will also have $xd=xa\wedge xb$ and $xm=xa\vee xb$ for every $x\in B_n^+$.

Notice that one can also define a {\it suffix} order $\succcurlyeq$, which is invariant under right-multiplication.

The key point in Garside work is to show, by elementary arguments, that $\sigma_i$ and $\sigma_j$ do have least common multiples in $B_n^+$. Namely
$$
    \sigma_i\vee \sigma_j = \left\{\begin{array}{ll}
    \sigma_i\sigma_j & \mbox{ if } |i-j|>1, \\
   \sigma_i\sigma_j\sigma_i  & \mbox{ if } |i-j|=1.
    \end{array}\right.
$$
He shows, at the same time, that $B_n^+$ is cancellative, that is, $xay=xby$ implies $a=b$ for every $a,b,x,y\in B_n^+$.

As the relations in presentation (\ref{E:presentation}) are homogenous, equivalent words in $B_n^+$ have the same length, so there is a well defined length in $B_n^+$. Although Garside does not mention it explicitly, induction on this length, together with the cancellativity condition, allows to show from the above result that every two elements in $B_n^+$ admit unique least common multiples and greatest common divisors.

Garside then studies the special element
$$
 \Delta=\sigma_1(\sigma_2\sigma_1)\cdots (\sigma_{n-1}\sigma_{n-2}\cdots \sigma_1).
$$
Again using elementary arguments, he shows that $\Delta=\sigma_1\vee \sigma_2\vee \cdots \vee \sigma_{n-1}$ in $B_n^+$, and that $\sigma_i\Delta= \Delta \sigma_{n-i}$ for $i=1,\ldots,n-1$. This implies that $\sigma_1,\ldots,\sigma_{n-1}$ are also {\it suffixes} of $\Delta$, that $\Delta^2$ commutes with every element of $B_n^+$ and, by induction on the length, that for every $a\in B_n^+$ one has $a\preccurlyeq \Delta^m$ and $\Delta^m\succcurlyeq a$ for some $m\geq 0$.

This has important implications. As every two elements in $B_n^+$ have a common multiple (some power of $\Delta$), and $B_n^+$ is cancellative, Ore's condition~\cite{Ore} says that $B_n^+$ embeds in its group of fractions. This group of fractions, due to presentation (\ref{E:presentation}) is precisely $B_n$. Therefore, $B_n^+$ is not only an algebraically defined monoid, but it can be considered as the subset of $B_n$ formed by the braids which can be expressed as words involving positive powers of the generators. These are called {\it positive} braids, and $B_n^+$ is called the monoid of positive braids.

One remark: The partial order $\preccurlyeq$ should not be confused with Dehornoy's ordering (which is a total order). Nor should the monoid of positive braids $B_n^+$ be confused with the semigroup $P$ of positive elements in Dehornoy's setting.

The above properties imply that the partial order $\preccurlyeq$ (respectively $\succcurlyeq$) can be extended to $B_n$ in the following way: $a\preccurlyeq b$ (resp. $b\succcurlyeq a$) if and only if $ac=b$ (resp. $b=ca$)  for some $c\in B_n^+$. This gives a partial order which is invariant under left-multiplication (resp. right-multiplication), and which admits unique least common multiples and greatest common divisors.

This structure allows to show lots of good properties of braid groups, as we will now see.

\subsection{Solution to the word problem}

Garside gave a new solution to the word problem in braid groups in the following way. Recall that for every $i=1,\ldots,n-1$ one has $\sigma_i\preccurlyeq \Delta$, that is, $\Delta=\sigma_iX_i$ for some $X_i\in B_n^+$. Given a braid written as a word in $\sigma_1,\ldots,\sigma_{n-1}$ and their inverses, one can replace each appearance of $\sigma_i^{-1}$ by $X_i\Delta^{-1}$. Conjugating a positive braid with $\Delta^{-1}$ gives a positive braid, so we can move all appearances of $\Delta^{-1}$ to the left. This shows that every braid can be written as $\Delta^p A$, for some $p\in \mathbb Z$ and some $A\in B_n^+$.
Moreover, if $\Delta\preccurlyeq A$, we can replace $\Delta^p$ by $\Delta^{p+1}$ and $A$ by $\Delta^{-1}A$. This decreases the length of $A$, so it can only be made a finite number of times. Therefore, every braid can be written, {\it in a unique way}, as $\Delta^p A$, where $p\in \mathbb Z$, $A\in B_n^+$ and $\Delta\not\preccurlyeq A$.

This {\it normal form} allows to solve the word problem, since one can enumerate all positive words representing the positive braid $A$, by iteratively applying the braid relations in every possible way. This was the solution given by Garside. It is not quite satisfactory, since it gives a highly inefficient algorithm.

Elrifai and Morton~\cite{Elrifai-Morton} (see also the work by Thurston~\cite{Epstein}) improved this by defining the {\it left normal form} of a braid. One just needs to take the decomposition $\Delta^p A$ as above, and then define $a_1=A\wedge \Delta$ and, by recurrence, $a_i=(a_{i-1}^{-1}\cdots a_1^{-1}A)\wedge \Delta$. In this way, every braid is written in a unique way as:
$$
    \Delta^p a_1\cdots a_r,
$$
where $a_i$ is a positive proper prefix of $\Delta$, that is $1\prec a_i \prec \Delta$, and also $(a_ia_{i+1})\wedge \Delta =a_i$, for $i=1,\ldots,r$. This is called the left normal form of the braid. Positive prefixes of $\Delta$ are called {\it simple elements} or {\it permutation braids}. Hence the normal form of a braid is a unique decomposition as a product of a power of $\Delta$ and a sequence of proper simple elements. Thurston~\cite{Epstein} showed that this normal form can be computed in time $O(l^2n\log n)$ if the input braid in $B_n$ is given as a word with $l$ letters.

\subsection{Braid groups are torsion-free: Fourth proof}

The Garside structure of the braid group allows to give a very simple proof that $B_n$ is torsion-free. Notice that this holds for every Garside group, including Artin-Tits groups of spherical type. We learnt this proof from John Crisp.

\begin{proposition}
Braid groups are torsion-free.
\end{proposition}

\begin{proof}
Let $x\in B_n$ and suppose that $x^n=1$ for some $n>0$. Consider the element $d= 1\wedge x \wedge \cdots \wedge x^{n-1}$. Then $xd=x(1\wedge x\wedge \cdots \wedge x^{n-1})= x\wedge x^2\cdots \wedge x^{n-1}\wedge 1 = d$. Canceling $d$, one has $x=1$.
\end{proof}

\subsection{The center of the braid groups}

As we mentioned above, Garside proved in a simple way that the center of $B_n$, when $n>2$, is the cyclic subgroup generated by $\Delta^2$, where
$$
\Delta = \sigma_1 (\sigma_2\sigma_1)(\sigma_3\sigma_2\sigma_1)\cdots (\sigma_{n-1}\sigma_{n-2}\cdots \sigma_1).
$$
This was already shown by Chow~\cite{Chow}, in a short paper in which he also shows the correctness of the presentation (\ref{E:presentation}). But we will give Garside's proof for its simplicity. Notice that $B_2\simeq \mathbb Z$, so the center of $B_2$ is equal to $\langle \Delta \rangle =\langle \sigma_1\rangle$. For all other cases we have:

\begin{theorem}\cite{Garside}
If $n>2$, the center of $B_n$ equals $\langle \Delta^2\rangle$.
\end{theorem}

\begin{proof}
Take an element of the center of $B_n$ written as $\Delta^p A$, where $A$ is positive and $\Delta\not\preccurlyeq A$. We first show that $A=1$.

Suppose that $A\neq 1$. Then $\sigma_i\preccurlyeq A$ for some $i$. If $p$ is even, $\Delta^p$ belongs to the center of $B_n$, hence so does $A$. Given $j$ such that $|i-j|=1$ one has $A\sigma_j\sigma_i = \sigma_j\sigma_i A$. As $\sigma_i$ and $\sigma_j$ are then both prefixes of $\sigma_j\sigma_i A$, their least common multiple must also be a prefix: $\sigma_j\sigma_i\sigma_j \preccurlyeq \sigma_j \sigma_i A$. Canceling $\sigma_j\sigma_i$ from the left,  one gets $\sigma_j\preccurlyeq A$. We can then iterate this process to obtain that $\sigma_1,\ldots,\sigma_{n-1}$ are all prefixes of $A$, hence $\Delta\preccurlyeq A$, a contradiction.

Suppose now that $p$ is odd. Then $(\Delta^p A)\sigma_{n-j}\sigma_{n-i} =\sigma_{n-j}\sigma_{n-i} (\Delta^p A) = \Delta^p \sigma_j\sigma_i A$. Hence $A\sigma_{n-j}\sigma_{n-i} =\sigma_{j}\sigma_{i}A$. If $|i-j|=1$, this implies as above that  $\sigma_j\sigma_i\sigma_j\preccurlyeq \sigma_j\sigma_i A$, hence $\sigma_j\preccurlyeq A$, and finally $\Delta\preccurlyeq A$, a contradiction.

Therefore, every element in the center of $B_n$ must be $\Delta^p$ for some $p$. Recall that $\Delta$ conjugates $\sigma_i$ to $\sigma_{n-i}$ for $i=1,\ldots,n-1$. This implies that odd powers of $\Delta$ belong to the center of $B_n$ if and only if $n=2$, hence this center consists of even powers of $\Delta$ when $n>2$.
\end{proof}

We remark that the same argument can be used to show that the center of an Artin-Tits group of spherical type is either $\langle\Delta\rangle$ or $\langle \Delta^2\rangle$, where $\Delta$ is the least common multiple of the standard generators~\cite{Brieskorn-Saito}.  In a general Garside group, the Garside element $\Delta$ is not necessarily the least common multiple of the standard generators (called atoms), so the same argument cannot be applied.

\subsection{Conjugacy problem}

The conjugacy decision problem in a group $G$ asks for an algorithm such that, given two elements $x,y\in G$, determines whether $x$ and $y$ are conjugate. The conjugacy search problem in $G$, on the other hand, asks for an algorithm such that, given two conjugate elements $x,y\in G$, finds a conjugating element. That is, finds $c\in G$ such that $c^{-1}xc=y$. Both problems are usually addressed to under the common name of {\it conjugacy problem}, as many algorithms solve both problems at the same time.

The conjugacy problem in braid groups was solved for the first time by Garside~\cite{Garside}, using what we now call the Garside structure of $B_n$. There have been several improvements of this algorithm, all of them using these techniques pioneered by Garside\cite{Elrifai-Morton,Birman-Ko-Lee1,Franco-GM1,Gebhardt,Birman-Gebhardt-GM1,Gebhardt-GM2, Gebhardt-GM3}. There is also a completely different solution: Charney~\cite{Charney} showed that braid groups (and Artin-Tits groups of spherical type) are biautomatic, and this yields an alternative solution to the conjugacy problem~\cite{Epstein}.

From all these solutions, the easiest to explain, and also the most efficient, is the one explained in detail in~\cite{Gebhardt-GM3} using the theoretical results from~\cite{Gebhardt-GM2}. For a short explanation of a simpler version of this algorithm, we recommend~\cite{Paris_2009}.

Basically, the algorithm goes as follows. There is a special kind of conjugation that can be applied to a braid, called {\it cyclic sliding}. Namely, if a braid $x$ is given in left normal form,
$$
    x=\Delta^p x_1\cdots x_r,
$$
then applying a cyclic sliding to $x$ consists of conjugating $x$ by its {\it preferred prefix} $\mathfrak p(x)=(\Delta^{p} x_1 \Delta^{-p})\wedge (x_r^{-1}\Delta)$. (If $r=0$ the preferred prefix is trivial.) We denote the resulting braid $\mathfrak s(x)$. This is a natural definition as it is the kind of operation that one usually uses to compute a left normal form, but applied to $x$ as if it was written as a cyclic word (`around a circle').

By iterated application of cyclic sliding, one eventually reaches a periodic orbit, that we call a {\it sliding circuit}. That is, if $\mathfrak s^{k}(x)=\mathfrak s^{t}(x)$ for some $k<t$, the sliding circuit of $x$ is $\{\mathfrak s^{k}(x),\mathfrak s^{k+1}(x),\ldots,\mathfrak s^{t-1}(x)\}$, a set of conjugates of $x$.

But in the conjugacy class of $x$ there can be, a priori, several sliding circuits. We then define $SC(x)$ as the {\it set of sliding circuits} in the conjugacy class of $x$. In other words, $SC(x)$ is the set of elements in periodic orbits for $\mathfrak s$ in the conjugacy class of $x$:
$$
  SC(x)=\{z\in B_n;\ \mathfrak s^{t}(z)=z \mbox{ for some $t>0$, and $z$ conjugate to $x$}\}.
$$
It is clear by definition that two braids $x$ and $y$ are conjugate if and only if $SC(x)=SC(y)$, which happens if and only if $SC(x)\cap SC(y)\neq \emptyset$.

The algorithm to solve the conjugacy problem in $B_n$ (and in every Garside group) described in~\cite{Gebhardt-GM3} does the following: Given $x,y\in B_n$, on one hand it computes $SC(x)$, and on the other hand it applies iterated cyclic sliding to $y$ until a repeated element $\tilde y$ is obtained. Then $x$ and $y$ are conjugate if and only if $\tilde y\in SC(x)$.

The computation of $SC(x)$ is based on the following results: First, $SC(x)$ is a finite set. Second, every two elements in $SC(x)$ are connected through a finite sequence of conjugations by {\it simple} elements, such that all the braids appearing along the way also belong to $SC(x)$. Third, the set of simple elements is finite. Therefore, one can compute the whole set $SC(x)$ by conjugating every known element by all simple elements, keeping those belonging to a sliding circuit, until no new element is obtained. Keeping track of the conjugations that connect the elements in $SC(x)$, this solves not only the conjugacy decision problem, but also the conjugacy search problem.

The solution just explained to compute $SC(x)$ is not efficient at all, of course, as the set of simple elements has $n!$ elements. Better ways to perform this task are explained in~\cite{Gebhardt-GM3}. We remark that this solution to the conjugacy problem holds for every Garside group (of finite type).

\subsection{Computations in braid groups}

For those interested in performing some computations in braids groups, Artin-Tits groups of spherical type, or other Garside groups, we can give some references.

Concerning braid groups, a web applet for computing left normal forms, solving the conjugacy problem, finding generators for the centralizer of a braid, and some other features can be found in~\cite{knotinfo}. The source code in C++ can be downloaded from~\cite{GMweb}.

The MAGMA computational algebra system also makes this kind of computations in braid groups~\cite{MAGMA}.

Finally, one can compute not only in braid groups, but in every Garside group using the CHEVIE package for GAP3~\cite{CHEVIE}.

\section{Braid groups are linear}

It is not possible to write a note on basic properties of braid groups without mentioning that braid groups are {\it linear}. This was a long-standing open question until it was solved for $B_4$ by Krammer~\cite{Krammer1}, and then for general $B_n$ using topological methods by Bigelow~\cite{Bigelow}. Shortly after that Krammer generalized his algebraic approach to all $B_n$~\cite{Krammer2}.

\begin{theorem}\cite{Krammer1,Bigelow,Krammer2}
There exists a faithful representation
$$
    \rho:\ B_n \rightarrow GL_N(k),
$$
where $k=\mathbb Q(p,q)$, and $N=n(n-1)/2$.
\end{theorem}

This result has been generalized to all Artin-Tits groups of spherical type~\cite{Cohen-Wales,Digne_2003,Paris_2002} (see also~\cite{Hee}).

Among the many consequences of this fact, it follows that Artin-Tits groups of spherical type are residually finite, and since they are finitely generated, they are also hopfian.

From these faithful representations of Artin-Tits groups, Marin deduced in a very short note~\cite{Marin} that pure Artin-Tits groups are residually torsion-free nilpotent. This implies, in particular, that pure Artin-Tits groups are bi-orderable (admit a total order which is invariant under left and right multiplication).

\section{Nielsen-Thurston classification}

We have seen that braids can be considered as automorphisms of the punctured disc and, modulo the center $\langle \Delta^2\rangle$ (modulo Dehn twists along the boundary), they can also be considered as automorphisms of the punctured sphere. In this setting, we can apply Nielsen-Thurston  theory~\cite{Thurston} to braids, in the following way: We can say that a braid in $B_n$ is periodic, reducible or pseudo-Anosov, if so is its projection onto $B_n/\langle \Delta^2\rangle$.

Let us briefly discuss about the above classification. Firstly, a braid is {\bf periodic} if, modulo the center of $B_n$, has finite order. That is, $\alpha\in B_n$ is periodic if and only if $\alpha^m=\Delta^p$ for some $m,p\in \mathbb Z$, $m\neq 0$. Recall from Section~\ref{SS:finite_order} that, after K\'er\'ekjart\'o and Eilenberg's result, this is equivalent to say that $\alpha$ is conjugate either to a power of $\delta=\sigma_1\cdots \sigma_{n-1}$ or to a power of $\epsilon=\sigma_1(\sigma_1\cdots \sigma_{n-1})$.

A braid $\alpha$ is {\bf pseudo-Anosov}~\cite{Thurston} if there exist two transverse measured foliations of the punctured disc, called the {\it unstable} ($\mathcal F_u$) and the {\it stable} ($\mathcal F_s$) foliations, respectively, such that $\alpha$ preserves $\mathcal F_u$ scaling the measure by a real number $\lambda >1$, and it also preserves $\mathcal F_s$, scaling the measure by $\lambda^{-1}$. These foliations admit some singular points, and are uniquely determined by $\alpha$ (up to so called Whitehead moves). The real number $\lambda$ is also determined by $\alpha$ and is called the {\it dilatation factor} of $\alpha$.

A braid $\alpha$ is said to be {\bf reducible} if it preserves a family of disjoint, non-degenerate, simple closed curves in the punctured disc $\mathbb D_n$.
Non-degenerate means that cannot be shrunk to a point (or a puncture), and that it is not isotopic to the boundary. In other words, that it encloses more than one  and less than $n$ punctures. For instance, the braid $\sigma_1\sigma_3\sigma_2\sigma_2$ preserves the family of curves depicted in Figure~\ref{Figure_10}. Notice that periodic braids can also be reducible: For instance $\Delta^2$ preserves every possible family of disjoint simple closed curves.

\begin{figure}[ht]
\centerline{\includegraphics{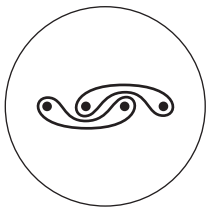}} \caption{This family of curves is preserved by the braid $\sigma_1\sigma_3\sigma_2\sigma_2$. }
\label{Figure_10}
\end{figure}

The famous theorem by Thurston of classification of surface homeomorphisms, applied to the case of braids, can be stated as follows:

\begin{theorem}\cite{Thurston}
Every braid belongs to exactly one of the following geometric types: (1) periodic, (2) pseudo-Anosov or (3) non-periodic and reducible.
\end{theorem}

In the reducible case, the invariant family of curves may be used to decompose (or reduce) the homeomorphism of a surface, this is why the curves are called {\it reduction curves}. Basically, this means cutting the surface along the reduction curves, and looking how the homeomorphism acts on each of the resulting components.  In the particular case of braids, there is an advantage with respect to the general case: The decomposition of a punctured disc along simple closed curves yields punctured discs, so the decomposition of a braid along reduction curves yields simpler braids. See~\cite{GM6} for a precise definition of this decomposition, in the case of braids.

There is also a natural way to see this decomposition, when the braid is represented as a collection of strands. Suppose that $\alpha$ is a braid that preserves a family of curves $\mathcal C$, as in the left hand side of Figure~\ref{Figure_11}. Notice that if we conjugate $\alpha$ by a braid $\eta$, we obtain a braid $\beta=\eta^{-1}\alpha\eta$ which preserves the family of curves $\eta(\mathcal C)$. Notice also that we can choose $\eta$ so that $\eta(\mathcal C)$ is isotopic to family of circles (in the geometric sense: points at the same distance of a given point called centre), see Figure~\ref{Figure_11}. Therefore, up to conjugacy, we can suppose that $\alpha$ preserves a family of circles.

\begin{figure}[ht]
\centerline{\includegraphics{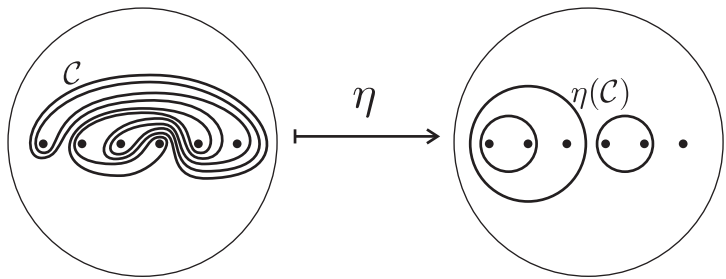}} \caption{If the family of curves $\mathcal C$ is preserved by a braid $\alpha$, the family $\eta(\mathcal C)$ is preserved by $\beta=\eta^{-1}\alpha\eta$. }
\label{Figure_11}
\end{figure}

If we now consider $\alpha$ as a motion of the $n$ base points, we can consider at the same time a motion {\it of the circles}, which go back to their original positions, possibly permuted. As the motion of the points describes strands, the motion of the circles describe {\it tubes}, that can be considered as fat strands. The outermost tubes, together with the strands which do not lay inside any tube, represent the {\it external braid}, while the braids inside the tubes represent the {\it interior braids}. These interior braids can be themselves reducible (we can have nested tubes), so they may also have an external braid and interior braids. Iterating this process, we decompose $\alpha$ along a family of curves, where each component is a braid with fewer strands.

As we saw above, the family of reduction curves of a given braid is not unique, so the above decomposition is not, a priori uniquely determined by $\alpha$. Moreover, this decomposition is defined up to conjugacy. Nevertheless, there is a well defined {\it canonical reduction system}~\cite{Birman-Lubotzky-McCarthy}, $CRS(\alpha)$, which makes the previous decomposition unique. It is precisely the set of {\it essential} curves for the braid $\alpha$, where we say that a curve $C$ is essential for $\alpha$ if:
\begin{enumerate}

 \item $\alpha^m (C)$ is isotopic to $C$ for some $m>0$.

 \item If $\alpha^m(C')$ is isotopic to $\mathcal C'$ for some $m>0$ and some simple
closed curve $C'$, then $C'$ can be isotoped to be disjoint to $C$.

\end{enumerate}

On the other hand, Lee and Lee~\cite{Lee-Lee} showed that there is a unique minimal positive braid sending any family of disjoint simple closed curves in $\mathbb D_n$ to a family of circles. Using this {\it minimal standardizer}, we can drop the expression {\it up to conjugacy} from the definition of the components of $\alpha$.

It is worth mentioning that $CRS(\alpha)\neq \emptyset$ if and only if $\alpha$ is reducible and non periodic, which is the third geometric type in Thurston's theorem.  By the above argument, we then have that the components of a reducible, non-periodic braid $\alpha$ are braids with fewer strands, which are uniquely determined by $\alpha$. Moreover, by Thurston~\cite{Thurston} (see also~\cite{Birman-Lubotzky-McCarthy}), all the components of $\alpha$ with respect to $CRS(\alpha)$ are either periodic or pseudo-anosov.

This geometric decomposition of braids is very useful for showing theoretical and also practical results. From the practical point of view, there is a project explained in~\cite{Birman-Gebhardt-GM1}, to try to solve the conjugacy problem in $B_n$ in polynomial time, which starts by distinguishing the periodic, pseudo-Anosov and reducible (non-periodic) case. In the periodic case, there is a polynomial solution given in~\cite{Birman-gebhardt-GM3}. Concerning theoretical results, we will end this note with a couple of examples.

\subsection{Centralizer of a braid}

The structure of the centralizer of a braid $\alpha$ depends heavily on its geometric type. In particular, the reduction curves of a braid $\alpha$ induce a semidirect product decomposition of its centralizer $Z(\alpha)$.

If $\alpha\in B_n$ is {\bf periodic}, we recall that it is conjugate to a rotation, that is, to a power of $\delta$ or $\epsilon$. The centralizer of these braids, determined in~\cite{Bessis-Digne-Michel} are precisely the braids which are invariant under that rotation. Algebraically, $Z(\alpha)$ is the braid group of an annulus, that is, the fundamental group of the configuration space of several points in an annulus (see~\cite{GM-Wiest}).

If $\alpha\in B_n$ is {\bf pseudo-Anosov}, one can use the well known result by McCarthy~\cite{McCarthy} applied to braids~\cite{GM-Wiest}, to see that $Z(\alpha)$ is isomorphic to $\mathbb Z\times \mathbb Z$. One generator is pseudo-Anosov, usually a root of $\alpha$ (but not always); the other generator is periodic, and can be chosen to be a root of $\Delta^2$.

Finally, if $\alpha\in B_n$ is reducible, its external braid is $\alpha_{ext}$, and its interior braids are $\alpha_1,\ldots,\alpha_t$, then in~\cite{GM-Wiest} it is shown that
$$
  Z(\alpha)\simeq \left(Z(\alpha_1)\times \cdots \times Z(\alpha_t)\right)\rtimes Z_0(\alpha_{ext}),
$$
where $Z_0(\alpha_{ext})$ is a subgroup of $Z(\alpha_{ext})$ depending only on the permutation induced by $\alpha_{ext}$. We then see that every element in the centralizer of $\alpha$ is {\it composed} by an element in the centralizer of $\alpha_{ext}$ and elements in the centralizers of $\alpha_1,\ldots,\alpha_t$. Notice that some $\alpha_i$ may be reducible and non-periodic, so $Z(\alpha_{i})$ may also be decomposed as a semidirect product as above.

This shows how the geometric classification of a braid determines some algebraic properties like the structure of its centralizer.

\subsection{Roots of a braid}

We end this note with a result that shows how the geometric techniques mentioned in this Section allow to show some theoretical results. We will care about the problem of unicity of roots in the braid group.

Let us first define a length in the braid group $B_n$. As the relations in Presentation~(\ref{E:presentation}) are homogeneous, given a braid $\alpha\in B_n$ we can define $s(\alpha)\in \mathbb Z$ to be the sum of the exponents in any writing of $\alpha$ as a word in the generators and their inverses.

Concerning unicity of roots in $B_n$, we have the following result in the pseudo-Anosov case.

\begin{proposition}\cite{GM4}
For $k\neq 0$, the $k$th root of a Pseudo-Anosov braid, if it exists, it is unique.
\end{proposition}

\begin{proof}
Let $\alpha$ and $\beta$ be two $k$-th roots of a pseudo-Anosov braid $\gamma$. That is $\alpha^k=\gamma=\beta^k$. The geometric type of a braid is preserved by taking powers, hence $\alpha$ and $\beta$ are also pseudo-Anosov. Moreover, the stable and unstable foliations are also preserved by taking powers, so $\alpha$, $\beta$ and $\gamma$ preserve the same pair of measured foliations. If $\lambda$ is the dilatation factor of $\alpha$, it is clear that $\lambda^{|k|}$ is the dilatation factor of $\gamma$, and then $\lambda$ is the dilatation factor of $\beta$.

Hence $\alpha$ and $\beta$ are pseudo-Anosov braids preserving the same pair of foliations and having the same dilatation factor. A priori, this does not necessarily imply that $\alpha=\beta$, but we will see that in $B_n$ this is the case.

First we show that $\alpha$ and $\beta$ commute. Consider $\rho=\alpha\beta\alpha^{-1}\beta^{-1}$. This braid preserves the same pair of invariant foliations, scaling their measures by 1. Therefore $\rho$ is a periodic element~\cite{ivanov}.  As periodic elements are conjugates of powers of $\delta$ and $\epsilon$, where $s(\alpha)=n-1$ and $s(\beta)=n$, the only periodic element with zero length is the trivial element, hence $\rho=1$, so $\alpha\beta=\beta\alpha$.

Therefore, from $\alpha^k=\beta^k$ we get $\alpha^k\beta^{-k}=1$. As $\alpha$ and $\beta$ commute, it follows that $(\alpha\beta^{-1})^k=1$, and since $B_n$ is torsion free, we finally obtain $\alpha\beta^{-1}=1$ and then $\alpha=\beta$.
\end{proof}

In the periodic case we do not have unicity of roots. For instance $(\sigma_1\sigma_2)^3=(\sigma_2\sigma_1)^3\in B_4$. But a conjecture by Makanin is satisfied: every two roots are conjugate.

\begin{proposition}\cite{GM4}
For $k\neq 0$, all $k$th roots of a periodic braid are conjugate.
\end{proposition}

\begin{proof}
If $\alpha$ and $\beta$ are $k$th roots of the same periodic braid $\gamma$, then $\alpha$ and $\beta$ are also periodic, and $s(\alpha)=s(\beta)$. Let $r=s(\alpha)=s(\beta)$. Recall that $\alpha$ and $\beta$ must be conjugate to a power of either $\delta$ or $\epsilon$. The length (exponent sum) of any conjugate of $\delta^m$ is $m(n-1)$. And the length of any conjugate of $\epsilon^t$ is $tn$.  Hence, if $r=m(n-1)$ and $m$ is not a multiple of $n$, both $\alpha$ and $\beta$ are conjugate to $\delta^m$. If $r=tn$ and $t$ is not a multiple of $n-1$, both $\alpha$ and $\beta$ are conjugate of $\epsilon^t$. It remains the case in which $r=ln(n-1)$, then $\alpha$ and $\beta$ could be conjugate to either $\delta^{nl}$ or to $\epsilon^{(n-1)l}$. As $\delta^n=\epsilon^{n-1}=\Delta^2$, these two braids above are precisely equal to $(\Delta^2)^l$, and also in this case $\alpha$ and $\beta$ are conjugate.
\end{proof}

From the pseudo-Anosov  and the periodic case, one can deduce that Makanin's conjecture holds for every braid, by decomposing the reducible braids and applying the above results to each of its components. See~\cite{GM4} for details of the reducible case. The result is then the following:

\begin{theorem}\cite{GM4}
For $k\neq 0$, all $k$th roots of a braid are conjugate.
\end{theorem}

This gives an idea on how these geometric tools can be used to show algebraic results in $B_n$.

\section{Acknowledgements}

I want to thank a number of people who made very interesting and useful comments and suggestions during the course at Pau: Vincent Florens, Ivan Marin, Luis Paris, Joan Porti, Michael Heusener, Arkadius Kalka, Andr\'es Navas, and those I am forgetting right now. Thanks also to the organizers of the School {\it Tresses in Pau}, for organizing such an interesting and pleasant event, and for {\it forcing} me to prepare this course, which I enjoyed so much. Thanks to the anonymous referee for a careful reading and many useful suggestions.

% The next command determines the bibliography style. Please do not
% change this.

\bibliographystyle{plain}

%  This inserts the bib file, biblio.bib
\bibliography{biblio}

% This command signals the end of the file.
\end{document}